\documentclass[11pt]{article}
\usepackage{graphicx}
\usepackage{amsmath,amsfonts,amssymb,latexsym,epsfig}
\usepackage{mathrsfs}
\usepackage{verbatim}
\usepackage{latexsym}
\usepackage{amsthm}
\usepackage{times}
\usepackage{color}
\usepackage{tikz}

\definecolor{darkred}{rgb}{.7,0,0}

\definecolor{darkgreen}{rgb}{0,0.5,0}

\definecolor{darkblue}{rgb}{0,0,0.7}

\newcommand{\bvp}{\overline{\vvarphi}}
\renewcommand{\epsilon}{\eps}

\newcommand{\Is}{\mathsf{I}}

\newcommand{\vvarphi}{\phi}

\newcommand{\Hm}{{\cal H}_\gamma}
\newcommand{\Hp}{{\cal H}_\gamma^{\perp}}

\definecolor{darkred}{rgb}{0.9,0.1,0.1}

\DeclareMathOperator{\argmin}{argmin\,}

\DeclareMathOperator{\Id}{Id}

\DeclareMathOperator{\spa}{span}
\DeclareMathOperator{\cov}{cov}

\newtheorem{theorem}{Theorem}[section]
\newtheorem{lemma}[theorem]{Lemma}
\newtheorem{remark}[theorem]{Remark}

\newtheorem{corollary}[theorem]{Corollary}
\newtheorem{proposition}[theorem]{Proposition}

\newtheorem{example}[theorem]{Example}
\newtheorem{definition}[theorem]{Definition}

\numberwithin{equation}{section}

\newcommand{\cA}{\mathcal A}
\newcommand{\cB}{\mathcal B}
\newcommand{\cC}{\mathcal C}
\newcommand{\cD}{\mathcal D}
\newcommand{\cF}{\mathcal F}
\newcommand{\cG}{\mathcal{G}}
\newcommand{\cH}{\mathcal{H}}

\newcommand{\cL}{\mathcal L}
\newcommand{\cN}{N}
\newcommand{\cNp}{N_{{P,0}}}

\newcommand{\E}{{\mathbb E}}
\newcommand{\N}{\mathbb{N}}
\newcommand{\R}{\mathbb{R}}
\newcommand{\Tn}{\mathbb{T}^n}

\newcommand{\eps}{\varepsilon}
\newcommand{\pph}{\vvarphi}
\newcommand{\ph}{\phi}

\newcommand{\Me}{\mathcal{M}(\cH)}
\newcommand{\En}{\E^\nu}

\newcommand{\Dkl}{D_{{\rm KL}}}
\newcommand{\DHel}{D_{{\rm hell}}}
\newcommand{\Dtv}{D_{{\rm tv}}}
\newcommand{\Dnm}{\Dkl(\nu \| \mu)}
\newcommand{\TC} { \mathcal{TC} (\cH)}
\newcommand{\HS}{\mathcal{HS}(\cH)}

\newcommand{\nmG}{\cNp(m,\Gamma)}
\newcommand{\nmsG}{\cNp(m_{\star},\Gamma_{\star})}
\newcommand{\Hnm}{H(\nu;\mu)}
\newcommand{\Cs}{C_\star}
\newcommand{\ns}{\nu_\star}
\newcommand{\ea}{e_\alpha}
\newcommand{\LHo}{\cL(\cH^1,\cH^{-1})}
\newcommand{\LHk}{\cL(\cH^{1-\kappa},\cH^{-(1-\kappa)})}
\newcommand{\Cni}{C_0^{-1}}
\newcommand{\Gn}{\Gamma_n}
\newcommand{\Gs}{\Gamma_\star}
\newcommand{\Hg}{\cH_\gamma}
\newcommand{\la}{\langle}
\newcommand{\ra}{\rangle}
\newcommand{\xa}{\xi_\alpha}
\newcommand{\nt}{\nu_t^{1\to 2}}
\newcommand{\ntg}{\nu_{t;\gamma}^{1 \to 2}}
\newcommand{\mg}{\mu_\gamma}
\newcommand{\tLg}{\tilde{\Lambda}_\gamma}
\newcommand{\tLgt}{\tilde{\Lambda}_{\gamma,t}}
\newcommand{\rTV}{\|_{{\rm tv}}}
\newcommand{\pg}{\pi_\gamma}

\begin{document}

\setlength{\baselineskip}{10pt} \title{Kullback-Leibler Approximation for
  Probability\\ Measures on Infinite Dimensional Spaces} \author{
  F.J. Pinski
  \footnote{E-mail address: frank.pinski@uc.edu} \\
  Physics Department\\
  University of Cincinnati\\
  PO Box 210011\\
  Cincinnati OH 45221, USA\\
  G. Simpson
\footnote{E-mail address: simpson@math.drexel.edu}\\
Department of Mathematics\\
Drexel University\\
Philadelphia, PA 19104 USA\\
  and\\
  A.M. Stuart and H. Weber
  \footnote{E-mail address: \{a.m.stuart,hendrik.weber\}@warwick.ac.uk.} \\
  Mathematics Institute \\
  Warwick University \\
  Coventry CV4 7AL, UK }
\maketitle

\begin{abstract}
In a variety of applications it is important to extract information from
a probability measure $\mu$ on an infinite dimensional space. Examples
include the Bayesian approach to inverse problems and 
possibly conditioned) continuous time Markov processes.
It may then be of interest to find a measure $\nu$, from within
a simple class of measures, which approximates $\mu$. This problem
is studied in the case where the Kullback-Leibler divergence
is employed to measure the quality of the approximation. A calculus
of variations viewpoint is adopted and the particular case where $\nu$ is
chosen from the set of Gaussian measures is studied in detail. Basic
existence and uniqueness theorems are established, together with
properties of minimising sequences. Furthermore, parameterisation
of the class of Gaussians through the mean and inverse covariance is 
introduced,  the need for regularisation is explained, and a
regularised minimisation is studied in detail. The 
calculus of variations framework resulting from this work
provides the appropriate underpinning for
computational algorithms.
\end{abstract}


\section{Introduction}
\label{sec:intro}

This paper is concerned with the problem of minimising the
Kullback-Leibler divergence between a pair of probability measures, viewed
as a problem in the calculus of variations. We are given a measure
$\mu$, specified  by its Radon-Nikodym derivative with respect to
a reference measure $\mu_0$, and we find the closest element $\nu$
from a simpler set of probability measures. After an initial study
of the problem in this abstract context, we specify to the situation
where the reference measure $\mu_0$ is Gaussian and the approximating
set comprises Gaussians. It is necessarily the case that minimisers
$\nu$ are then equivalent as measures to $\mu_0$ and we use the
Feldman-Hajek Theorem to characterise such $\nu$ in terms of
their inverse covariance operators. This induces a natural formulation
of the problem as minimisation over the mean, from the Cameron-Martin
space of $\mu_0$, and over an operator from a  weighted Hilbert-Schmidt space.
We study this problem from the point of view of the calculus
of variations, studying properties of minimising sequences,
regularisation to improve the space in which operator convergence is
obtained, and uniqueness under a slight strengthening of a log-convex
assumption on the measure $\mu.$ 

In the situation where the minimisation is over a convex set of 
measures $\nu$, the problem is classical and completely 
understood \cite{CS}; in particular, there is uniqueness of minimisers.
However, the emphasis in our work is on situations where the set of 
measures $\nu$ is not convex, such as the set of Gaussian measures, and in this
context uniqueness cannot be expected in general. However some of the ideas 
used in \cite{CS} are useful in our general developments, in particular
methodologies to extract minimising sequences converging in total variation. 
Furthermore, in the finite dimensional case the minimisation problem  at 
hand was studied by McCann \cite{McCann} in the context of  gas dynamics. 
He introduced the concept of ``displacement convexity" which was 
one of the main ingredients for the recent developments in the theory 
of mass transportation (e.g. \cite{AGS, Villani}).  
Inspired by the work of McCann, we identify situations in which
uniqueness of minimisers can occur even when approximating over
non-convex classes of measures.

In the study of inverse problems in partial differential equations, when given
a Bayesian formulation \cite{S10a}, and in the study of conditioned 
diffusion processes \cite{HSV11},
the primary goal is the extraction of information from a probabililty
measure $\mu$ on a function space. This task often requires computational
methods. One commonly adopted approach is to find the maximum a posteriori
(MAP) estimator which corresponds to identifying the centre of balls of
maximal probability, in the limit of vanishingly small radius 
\cite{DLSV13,kaipio2005statistical}; in the context of inverse problems
this is linked to the classical theory of Tikhonov-Phillips regularisation 
\cite{engl1996regularization}. Another commonly adopted approach is to
employ Monte-Carlo Markov chain (MCMC) methods \cite{liu2008monte} to sample
the probability measure of interest. The method of MAP estimation 
can be computationally tractable, but loses important
probabilistic information. In contrast MCMC methods can, in principle,
determine accurate probabilistic information but may be very expensive. 
The goal of this work is to provide the mathematical basis for computational
tools which lie between MAP estimators and MCMC methods. Specifically
we wish to study the problem of approximating the measure $\mu$ from
a simple class of measures and with quality of approximation measured
by means of the Kullback-Leibler divergence. This holds the potential
for being a computational tool which is both computationally tractable
and provides reliable probabilistic information. The problem leads to
interesting mathematical questions in the calculus of variations, and study of
these questions form the core of this paper.

Approximation with respect to Kullback-Leibler divergence is not new
and indeed forms a widely used tool in the field of machine learning
\cite{bishop2006pattern} with motivation being the interpretation of
Kullback-Leibler divergence as a measure of loss of information.
Recently the methodology has been
used for the coarse-graining of stochastic lattice systems
\cite{katsoulakis2007coarse}, 
simple models for data assimilation 
\cite{archambeau2007gaussian,archambeau2007variational}, 
the study of models in ocean-atmosphere science
\cite{majda2011improving, giannakis2012quantifying} 
and molecular dynamics \cite{katsoulakis2013information}. 
However none of this applied work has studied the underlying calculus
of variations problem which is the basis for the algorithms employed.
Understanding the properties of minimising sequences is crucial
for the design of good finite dimensional approximations, see
for example \cite{ball1987numerical}, and this fact motivates the work herein.
In the companion paper \cite{PSSW14} we will demonstrate the
use of algorithms for Kullback-Leibler minimisation 
which are informed by the analysis herein.

In section \ref{sec:setup} we describe basic facts about KL minimisation
in an abstract setting, and include an example illustrating our
methodology, together with the fact
that uniqueness is typically not to be expected when approximating 
within the Gaussian class. Section \ref{sec:abstract} then concentrates
on the theory of minimisation with respect to Gaussians. We demonstrate
the existence of minimisers, and then develop a regularisation theory needed 
in the important case where the inverse covariance operator is parameterised
via a Schr\"odinger potential. We also study the restricted class of target
measures for which uniqueness can be expected, and we generalize the
overall setting to the study of Gaussian mixtures. Proofs
of all of our results are collected in section \ref{sec:proofs},
whilst the Appendix contains variants on a number of classical
results which underlie those proofs.

\subsection*{}
\emph{Acknowledgments:} The work of AMS is supported by ERC, EPSRC and ONR.
GS was supported by NSF PIRE grant OISE-0967140 and DOE grant DE-SC0002085. 
Visits by FJP and GS to Warwick were supported by ERC, EPSRC and ONR.
AMS is grateful to Colin Fox 
for fruitful discussions on related topics.

\section{General Properties of KL-Minimisation}
\label{sec:setup}

In subsection \ref{ssec:BT} we present some basic background theory
which underpins this paper. In subsection \ref{ssec:Examples}
we provide an explicit finite dimensional example which serves to
motivate the questions we study in the remainder of the paper.

\subsection{Background Theory}
\label{ssec:BT}

In this subsection we recall some general facts about  Kullback-Leibler approximation on an arbitrary Polish space.
Let $\cH$ be a Polish space endowed with its Borel sigma algebra $\cF$. Denote by $\Me$ the set of Borel probability measures on $\cH$ and let $\cA \subset \Me$.  Our aim is to find the best approximation of a target measure $\mu \in \Me$  in the set $\cA$ of ``simpler" measures. 
As a measure for closeness we choose the Kullback-Leibler divergence, also
known as the relative entropy. For any $\nu \in \Me$ that is absolutely continuous with respect to $\mu$ it is given by 
\begin{equation}
\Dnm = \int_H \log \bigg( \frac{d \nu}{ d \mu}(x) \bigg)  \frac{d \nu}{ d \mu}(x) \, \mu(dx) = \E^\mu \bigg[  \log \bigg( \frac{d \nu}{ d \mu}(x) \bigg)  \frac{d \nu}{ d \mu}(x)   \bigg],\label{e:KL}
\end{equation}
where we use the convention that $0 \log 0 =0$.  If $\nu$ is not absolutely continuous with respect to $\mu$, then the Kullback-Leibler divergence is defined as $+\infty$. 
The main aims of this article are to discuss the properties of the minimisation problem 
\begin{equation}
\underset{\nu \in \cA} \argmin \Dnm \label{e:problem1}
\end{equation}
for  suitable sets $\cA$, and to create a mathematical framework appropriate
for the development of algorithms to perform the minimisation.

The Kullback-Leibler divergence is not symmetric in its arguments and minimising $\Dkl(\mu \| \nu)$ over $\nu$ for fixed $\mu$ in general gives a different result than \eqref{e:problem1}. Indeed, if $\cH$ is $\R^n$ and $\cA$ is the set of Gaussian measures on $\R^n$, then minimising $\Dkl(\mu \| \nu)$  yields for $\nu$ the Gaussian measure with the same mean and variance as $\mu$;
see \cite[section 10.7]{bishop2006pattern}. 
Such an approximation is undesirable in many situations,
for example if $\mu$ is bimodal; see \cite[Figure 10.3]{bishop2006pattern}. 
We will demonstrate by example in subsection \ref{ssec:Examples}
that problem \eqref{e:problem1} is a more desirable minimisation problem which
can capture local properties of the measure $\mu$ such as individual modes.
Note that the objective function in the minimisation \eqref{e:problem1} 
can formulated in terms of expectations only over measures from $\cA$; if 
this set is simple then this results in computationally expedient algorithms. 
Below we will usually chose for $\cA$ a set of Gaussian measures and hence these expectations are readily computable.

The following well-known result gives existence of minimisers for  problem  \eqref{e:problem1} as soon as the set $\cA$ is closed under weak convergence of probability measures.
 For the reader's convenience we give a proof in the  
Appendix. We essentially follow the exposition in \cite[Lemma 1.4.2]{DE};
see also \cite[Lemma 9.4.3]{AGS}.
  
\begin{proposition}\label{le:firstproperties}
Let $(\nu_n)$ and $(\mu_n)$ be sequences in $\Me$ that converge weakly to $\nu_\star$ and $\mu_\star$. Then we have
\begin{equation*}
\liminf_{n \to \infty} \Dkl(\nu_n \| \mu_n) \geq \Dkl(\ns \| \mu_\star).
\end{equation*} 
Furthermore, for any $\mu \in \Me$ and for any $M < \infty$ the set 
\begin{equation*}
\{ \nu \in \Me \colon \Dkl(\nu \|\mu ) \leq M\} 
\end{equation*}
is compact with respect to weak convergence of probability measures.
\end{proposition}

Proposition \ref{le:firstproperties} yields the following immediate corollary
which, in particular, provides the existence of minimisers from within the
Gaussian class:

\begin{corollary}\label{c:firstproperties}
Let $\cA$ be closed with respect to weak convergence. Then,
for given $\mu \in \Me$,  assume that there exists $\nu \in \cA$
such that $\Dnm<\infty.$ It follows that
there exists a minimiser $\nu \in \cA$
solving problem \eqref{e:problem1}. 
\end{corollary}

If we know in addition that the set $\cA$ is \emph{convex} then the following classical stronger result holds: 
\def \CS {\cite[Theorem 2.1]{CS}}
\begin{proposition}[\CS]\label{thm:CS}
Assume that  $\cA$ is convex and closed with respect to \emph{total variation} convergence. Assume furthermore that there exists a $\nu \in \cA$ with $\Dnm < \infty$. Then there exists a \emph{unique} minimiser $\nu \in \cA$
solving problem \eqref{e:problem1}. 
\end{proposition}

However in most situations of interest in this article, such as
approximation by Gaussians, the set $\cA$ is not convex. Moreoever,  the proof of Proposition \ref{thm:CS} does not carry over 
to the case of non-convex $\cA$ and, indeed, uniqueness of minimisers is
not expected in general in this case (see, however, the discussion of uniqueness in Section \ref{ssec:uniqueness}). Still, the methods used in proving  
Proposition \ref{thm:CS}  do have the following interesting consequence for our setting. Before we state it we recall the definition of the total variation norm of two probability measures. It is given by
\begin{align*}
\Dtv(\nu,\mu)= \| \nu - \mu \rTV &= \frac12 \int  \left| \frac{d \nu}{d \lambda}(x) - \frac{d \mu}{d \lambda}(x)\right|  \lambda(dx) \notag\\
\end{align*}
where $\lambda$ is a probability measure on $\cH$ such that $\nu \ll \lambda$ and $\mu \ll \lambda$

\begin{lemma}\label{le:CS}
Let $(\nu_n)$ be a sequence in $\Me$ and let $\ns \in \Me$ and $\mu \in \Me$ be probability measures such that for any $n \geq 1 $ we have  $\Dkl(\nu_n\|\mu)<\infty$ and $\Dkl(\ns\|\mu)<\infty$.  Suppose that the $\nu_n$ converge weakly to $\nu_{\star}$ and in addition that 
\begin{equation*}
\Dkl(\nu_n\| \mu) \to \Dkl (\ns\| \mu).
\end{equation*}
Then $\nu_n$ converges to $\ns$ in total variation norm. 
\end{lemma}

The proof of Lemma~\ref{le:CS} can be found in Section~\ref{sec:lemmaCS}.
Combining Lemma \ref{le:CS} with Proposition \ref{le:firstproperties}
implies in particular the following:
\begin{corollary}\label{cor:mini}
Let $\cA$ be closed with respect to weak convergence and $\mu$ such that there exists a $\nu \in \cA$ with $\Dnm < \infty$. Let $\nu_n \in \cA$ satisfy 
\begin{equation}\label{e:minimising}
\Dkl(\nu_n \| \mu) \to \inf_{\nu \in \cA} \Dnm.
\end{equation}
Then, after passing to a subsequence, $\nu_n$ converges weakly to a $\ns \in \cA$  that realises the infimum in \eqref{e:minimising}. Along the
subsequence we have, in addition, that
\begin{equation*}
\| \nu_n - \ns \rTV \to 0.
\end{equation*}
\end{corollary}

Thus, in particular, if ${\cal A}$ is the Gaussian class then the preceding
corollary applies.

\subsection{A Finite Dimensional Example}
\label{ssec:Examples}
In this subsection we illustrate the minimisation problem in the simplified situation where $\cH=\R^n$ for some $n\geq1$. In this situation it is natural to consider target measures $\mu$ of the form 
\begin{equation}\label{e:target2}
\frac{d \mu}{ d \cL^n}(x) = \frac{1}{Z_\mu}\exp \big( - \Phi(x) \big) ,
\end{equation}
for some smooth function $\Phi\colon \R^n \to \R_{+}$. Here $\cL^n$ denotes the Lebesgue  measure on $\R^n$.  We consider the minimisation
problem \eqref{e:problem1} in the case where $\cA$ is the set of all Gaussian measures on $\R^n$.

If $\nu = \cN(m,C)$ is a Gaussian on $\R^n$ with mean $m$ and a non-degenerate covariance matrix $C$ we get
\begin{align}
\Dnm &= \E^\nu\bigg[ \Phi(x) - \frac{\langle x,C^{-1}x \rangle }{2}  \bigg] - \frac12 \log \big( \det C\big) + \log \bigg(  \frac{Z_\mu}{(2\pi)^{\frac{n}{2}}}   \bigg) \notag\\
&= \E^\nu\big[ \Phi(x)\big] - \frac12 \log \big( \det C\big) - \frac{n}{2} + \log \bigg(  \frac{Z_\mu}{(2\pi)^{\frac{n}{2}}}   \bigg) \label{e:DKLfinitedim}.
\end{align}
The last two terms on the right hand side of \eqref{e:DKLfinitedim} do not depend on the Gaussian measure $\nu$ and can therefore be dropped in the minimisation problem. In the case where $\Phi$ is a polynomial the expression $\E^\nu\big[ \Phi(x)\big]$ consists of a Gaussian expectation of a polynomial and it can be evaluated explicitly. 

\begin{figure}
\begin{center}
\begin{tikzpicture}[xscale=0.03, yscale=4]

\draw[black, ->] (1,0) -- (201,0);
\draw[black, ->] (100,-0.2) -- (100,1.0);


\draw (40,0) node[below]{$-1$}; 
\draw (161,0) node[below]{$1$}; 

\draw (200,.8) node[right]{$\Phi$  };

\draw[black] plot file{DW.txt};

\end{tikzpicture}
\caption{The double well potential $\Phi$.}\label{fig:2}
\end{center}
\end{figure}
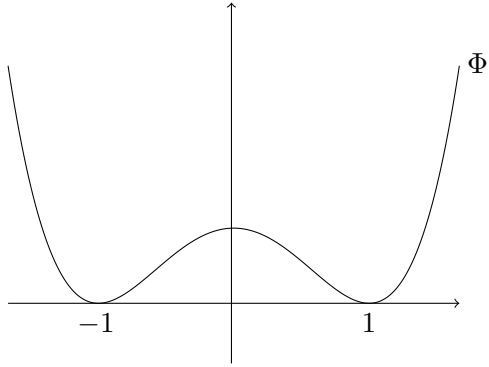

To be concrete we consider the case where $n=1$ and $\Phi(x) = \frac{1}{4 \eps} (x^2 -1)^2$  so that the measure $\mu$ has two peaks: see Figure \ref{fig:2}. 
In this one dimensional situation we minimise $\Dnm$ over all measures $\cN(m,\sigma^2),$  $m \in \R, \sigma \geq 0$.
Dropping the irrelevant constants in  \eqref{e:DKLfinitedim}, we are led to minimise
\begin{align}
\cD(m,\sigma) &:=  \E^{\cN(m,\sigma^2)}\big[ \Phi(x)\big] -  \log (  \sigma ) \notag\\
&= \Big( \Phi(m) + \frac{\sigma^2}{2} \Phi''(m)  +\frac{3 \sigma^4}{4!} \Phi^{(4)}(m)  \Big)   -  \log (  \sigma ) \notag\\
&= \frac{1}{\eps}\bigg( \frac{1}{4}(m^2-1)^2 +\frac{\sigma^2}{2} (3m^2-1)+ \frac{3 \sigma^4}{4} \bigg)   -  \log (  \sigma ). \notag
\end{align}

We expect, for small enough $\epsilon$, to find two different Gaussian
approximations, centred near $\pm 1.$ 
Numerical solution of the critical points of $\cD$ (see Figure \ref{fig:1}) 
confirms this intuition. In fact we see the existence of three, then five and finally one critical point
as $\epsilon$ increases. For
small $\epsilon$ the two minima near $x=\pm 1$ are the global minimisers,
whilst for larger $\epsilon$ the minimiser at the origin is the global minimiser.
\begin{figure}
\begin{center}
\includegraphics[width=10cm]{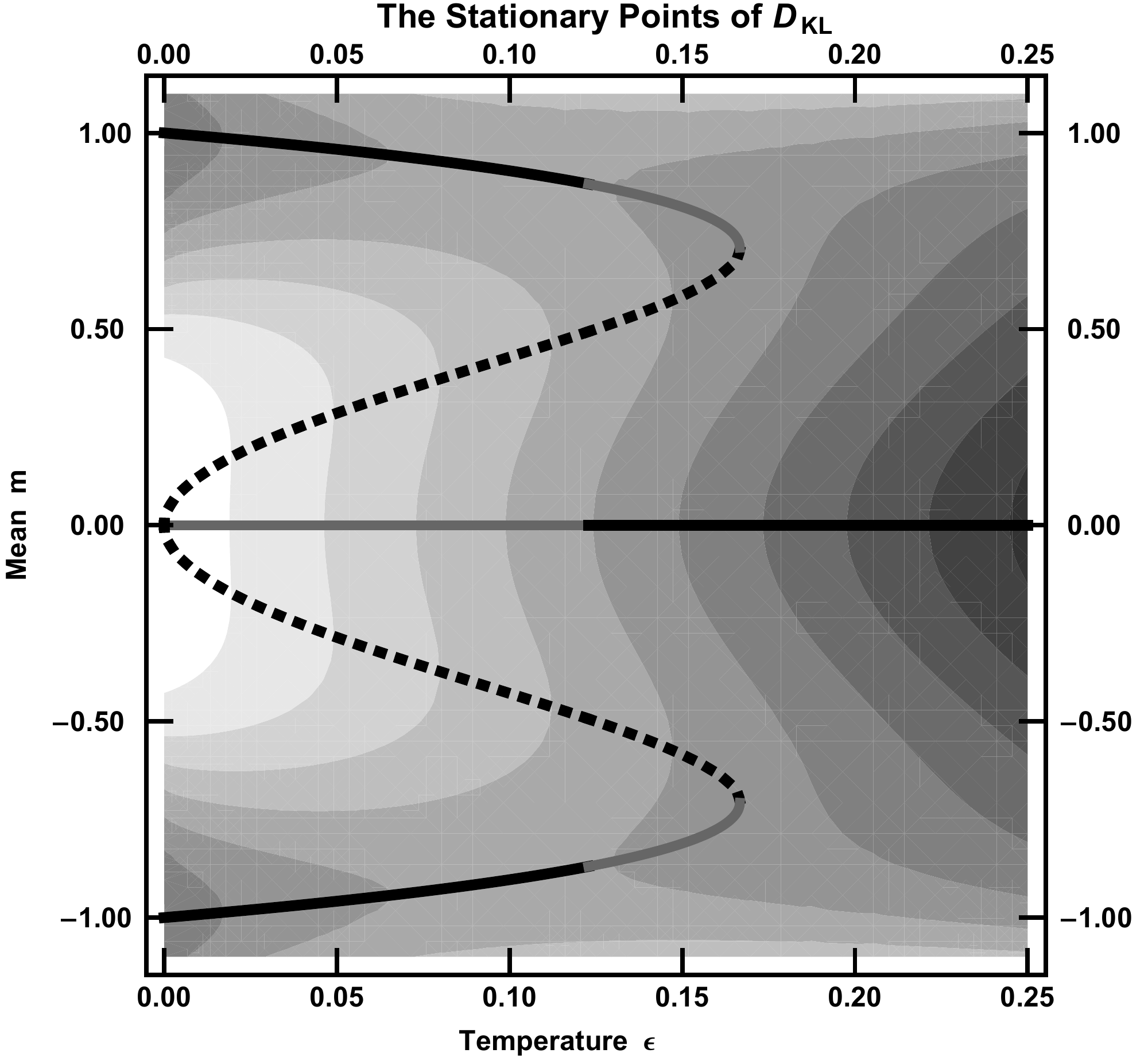}
\end{center}
\caption{In the figure solid lines denote minima, with the darker line used for the absolute minimum 
at the given temperature $\epsilon$. The dotted lines denote  maxima.
At $\epsilon=1/6$ two stationary points annihilate one another at a fold bifurction
and only the symmetric solution, with mean $m=0$, remains. However even for
$\epsilon > 0.122822$, the symmetric mean zero solution is the global minimum.}
\label{fig:1}
\end{figure}

\section{KL-Minimisation over Gaussian Classes}
\label{sec:abstract}

The previous subsection demonstrates that the class of Gaussian measures
is a natural one over which to minimise, although uniqueness cannot, in general,
be expected. In this section we therefore study approximation
within Gaussian classes, and variants on this theme. 
Furthermore we will assume that the measure of
interest, $\mu$, is  equivalent (in the sense of measures)
to a Gaussian $\mu_0=N(m_0,C_0)$ on the separable Hilbert space 
$(\cH, \la \cdot , \cdot \ra, \, \| \cdot \|)$, 
with $\cF$ the Borel $\sigma$-algebra. 

More precisely, let  $X\subseteq \cH$ be a separable Banach space which is 
continuously embedded in $\cH$,  where $X$ is measurable with 
respect to $\cF$ and satifies $\mu_0(X)=1$.
We also assume that $\Phi:X \to \R$ is continuous in the topology of $X$ 
 and that $\exp(-\Phi(x))$ is integrable with respect to $\mu_0$. 
\footnote{In fact continuity is only used in subsection \ref{ssec:uniqueness};
measurability will suffice in much of the paper.}
Then the target measure $\mu$ is defined by  
\begin{equation}
\frac{d \mu}{ d \mu_0} (x)= \frac{1}{Z_\mu} \exp \big( - \Phi(x)\big), \label{e:target}
\end{equation}
where the normalisation constant is given by 
\begin{equation}
Z_\mu = \int_{\cH} \exp \big( - \Phi(x)\big)\, \mu_0(dx) =: \E^{\mu_0} \big[  \exp \big( - \Phi(x)\big) \big].\notag
\end{equation}
Here and below we use the notation $\E^{\mu_0} $ for  the expectation with respect to the probability measure $\mu_0$, and we also use similar notation for 
expectation with respect to other probability measures.
Measures of the form (3.1) with $\mu_0$ Gaussian occur in the Bayesian
approach to inverse problems with Gaussian priors, and in the pathspace
description of (possibly conditioned) diffusions with additive noise.

In subsection \ref{ssec:GC} we recall some basic definitions concerning
Gaussian measure on Hilbert space and then state a straightforward
consequence of the theoretical developments of the previous section,
for ${\cal A}$ comprising various Gaussian classes. Then, in
subsection \ref{ssec:PGM}, we discuss how to parameterise the
covariance of a Gaussian measure, introducing Schr\"odinger potential-type
parameterisations of the precision (inverse covariance)
operator. By example we show that whilst Gaussian measures within
this parameterisation may exhibit well-behaved minimising sequences,
the potentials themselves may behave badly along minimising sequences,
exhibiting oscillations or singularity formation. This motivates
subsection \ref{ssec:REG} where we regularise the minimisation to
prevent this behaviour. In subsection \ref{ssec:uniqueness}
we give conditions on $\Phi$ which result in uniqueness of
minimisers and in subsection \ref{ssec:mixtures}  we make some
remarks on generalisations of approximation within the class of
Gaussian mixtures.

\subsection{Gaussian Case}
\label{ssec:GC}

We start by
recalling some basic facts about Gaussian measures. 
A probability measure $\nu$ on a separable Hilbert space $\cH$ is Gaussian if for any $\ph$ in the dual space $\cH^\star$ the push-forward measure $\nu \circ \ph^{-1}$ is Gaussian (where Dirac measures are viewed as Gaussians with variance $0$) \cite{ZDP}. 
Furthermore, recall that $\nu$ is characterised by its mean and covariance, 
defined via the following (in the first case Bochner) integrals: the mean $m$ is given by
\begin{equation*}
m := \int_{\cH} x \,  \nu(dx) \in \cH
\end{equation*}
and its covariance operator $C\colon \cH \to \cH$ satisfies
\begin{equation*}
\int_{\cH} \la x, y_1\ra \la x, y_2\ra  \, \nu(dx)  = \la y_1, C y_2 \ra,
\end{equation*} 
for all $y_1,y_2 \in \cH$. Recall that $C$ is a non-negative, symmetric, trace-class operator, or equivalently $\sqrt{C}$ is a non-negative, symmetric Hilbert-Schmidt operator. In the sequel we will denote by $\cL(\cH)$, $\TC$, and $\HS$ the spaces of linear, trace-class, and Hilbert-Schmidt operators on $\cH$. We denote the Gaussian measure with mean $m$ and covariance operator $C$ by $\cN(m,C)$. We have collected some additional facts about Gaussian measures in Appendix~\ref{A:Gaussian}.

From now on, we fix a Gaussian measure $\mu_0 = \cN(m_0,C_0)$. We always assume that $C_0$ is a strictly positive operator. We denote the image of  $\cH$ under $C_0^{\frac12}$, endowed with the scalar product $\la C_0^{-\frac12} \cdot, C_0^{-\frac12} \cdot \ra$, by $\cH^{1}$, noting that this is the \emph{Cameron-Martin space} of $\mu_0$; we denote its dual space by  $\cH^{-1} =\big( \cH^{1}\big)^{\star}$. We will make use of the natural finite dimensional projections associated to the operator $C_0$ in several places in the sequel and so we introduce 
notation associated with this for later use. Let $(\ea, \alpha \geq 1)$ be the basis of $\cH$ consisting of eigenfunctions of $C_0$, and let $(\lambda_\alpha, \alpha \geq 1)$ be the associated sequence of eigenvalues. For simplicity we assume that the eigenvalues are in non-increasing order. Then for any $\gamma \geq 1$ we will denote by $\Hg := \spa(e_1, \ldots , e_\gamma)$ and the orthogonal projection onto $\Hg$ by 
\begin{equation}\label{e:projectgamma}
\pg \colon \cH \to \cH, \qquad x \mapsto \sum_{\alpha =1}^{\gamma} \langle x , \ea \rangle \, \ea.
\end{equation}
 Given such a measure $\mu_0$ we assume  that the target measure $\mu$ is given by \eqref{e:target}.

For $\nu \ll \mu$ expression \eqref{e:KL} can be rewritten, using \eqref{e:target} and the equivalence of $\mu$ and $\mu_0$, as 
\begin{align}
\Dnm
&= \En \bigg[  \log \bigg( \frac{d \nu}{ d \mu}(x) \bigg) \mathbf{1}_{ \{ \frac{d \nu}{ d \mu} \neq 0 \}  } \bigg]   \notag\\
&= \En \bigg[  \log \bigg( \frac{d \nu}{ d \mu_0}(x) \times \frac{d\mu_0}{d\mu}(x) \bigg) \mathbf{1}_{ \{ \frac{d \nu}{ d \mu_0} \neq 0 \}  } \bigg]   \notag\\
&=  \En \bigg[  \log \bigg( \frac{d \nu}{ d \mu_0}(x) \bigg) \mathbf{1}_{ \big\{ \frac{d \nu}{ d \mu_0} \neq 0 \big\}}    \bigg] + \En\big[ \Phi(x) \big] +\log (Z_\mu) .\label{e:KL2}
\end{align}
The expression in the first line shows that in order to evaluate the Kullback-Leibler divergence it is sufficient to compute an expectation with respect to the approximating measure $\nu \in \cA$ and not with respect to the target $\mu$.

The same expression shows positivity. To see this decompose the measure $\mu$ into two non-negative measures $\mu = \mu^{\|}+ \mu^{\perp}$ where  $\mu^{\|}$ is equivalent to $\nu$ and $\mu^{\perp}$ is singular with respect to $\nu$. Then we can write  with the Jensen inequality
\begin{align*}
\Dnm =& - \En \bigg[  \log \bigg( \frac{d \mu^{\|}}{ d \nu}(x) \bigg)  \mathbf{1}_{ \{ \frac{d \nu}{ d \mu} \neq 0 \}  }  \bigg] \geq  - \log \En \bigg[   \frac{d \mu^{\|}}{ d \nu}(x)    \bigg] \\
=& - \log \mu^{\|} (\cH) \geq 0.
\end{align*}
This establishes the general fact that relative entropy is non-negative
for our particular setting.

Finally, the expression in the third line of \eqref{e:KL2} shows that the normalisation constant $Z_\mu$ enters into $\Dkl$ only  as an additive constant that can be ignored in the minimisation procedure. 

If we assume furthermore, that the set $\cA$ consists of Gaussian measures, Lemma~\ref{le:CS} and Corollary~\ref{cor:mini} imply the following result.
\begin{theorem}\label{thm:Gaussian}
Let $\mu_0$ be a Gaussian measure with mean $m_0 \in \cH$ and covariance operator $C_0 \in \TC$ and let $\mu$ be given by \eqref{e:target}. Consider the following choices for $\cA$ 
\begin{enumerate}
\item $\cA_1 = \{ \text{Gaussian measures on } \cH \}$,
\item $\cA_2 = \{ \text{Gaussian measures on } \cH \text{ equivalent to } \mu_0 \}$,
\item For a fixed covariance operator $\hat{C} \in \TC$ 
\begin{equation*}
 \cA_3 = \{ \text{Gaussian measures on } \cH \text{ with covariance } \hat{C} \},
\end{equation*}
\item For a fixed mean $\hat{m} \in \cH$
\begin{equation*}
 \cA_4 = \{ \text{Gaussian measures on } \cH \text{ with mean } \hat{m} \}.
\end{equation*}
\end{enumerate}
In each of these situations, as soon as there exists a single $\nu \in \cA_i$ with $\Dnm < \infty$ there exists a minimiser of $\nu \mapsto \Dnm$ in $\cA_i$. 
Furthermore $\nu$ is necessarily equivalent to $\mu_0$ in the sense of measures.
\end{theorem}
\begin{remark}
\label{rem:32}
Even in the case $\cA_1$ the condition that there exists a single $\nu$ with finite $\Dnm$ is not always satisfied. For example, if $\Phi(x) = \exp\big( \| x\|_{\cH}^4\big)$ then for  any Gaussian measure $\nu$ on $\cH$ we have, using
the identity \eqref{e:KL2}, that
\begin{equation*}
\Dnm = \Dkl(\nu \| \mu_0)    + \En\big[ \Phi(x) \big] +\log (Z_\mu) = +\infty.\label{e:KL2B}
\end{equation*}
In the cases $\cA_1, \cA_3$ and $\cA_4$ such a $\nu$ is necessarily
absolutely continuous with respect to $\mu$, and hence equivalent to $\mu_0$;
this equivalence is encapsulated directly in $\cA_2.$
The conditions for this to be possible are stated in the Feldman-Hajek Theorem, Proposition \ref{prop:FH}.
\end{remark}

\subsection{Parametrization of Gaussian Measures}
\label{ssec:PGM}

When solving the minimisation
problem \eqref{e:problem1} it will usually be convenient to parametrize the set $\cA$ in a suitable way. In the case where $\cA$ consists of all Gaussian measures on $\cH$ the first choice that comes to mind is to parametrize it by the mean $m \in \cH$ and the covariance operator $C  \in \TC$. 
In fact it is often convenient, for both computational and modelling reasons,
to work with the inverse covariance (precision) operator which, because
the covariance operator is strictly positive and trace-class, 
is a densely-defined unbounded operator.

Recall that the underlying Gaussian reference measure $\mu_0$ has 
covariance $C_0$. We will consider covariance operators $C$  of the form
\begin{equation}\label{e:defGamma}
C^{-1} = \Cni + \Gamma,
\end{equation}
for suitable operators $\Gamma$.  From an applications perspective it is interesting to consider the
case where $\cH$ is a function space and $\Gamma$ is
a mutiplication operator. Then $\Gamma$ has the form
$\Gamma u =v(\cdot) u(\cdot)$ for some fixed function $v$ which we refer
to as a {\em potential} in analogy with the Schr\"odinger setting. In this case parametrizing the Gaussian family
$\cA$ by the pair of functions $(m,v)$ comprises a considerable dimension 
reduction over parametrization by the pair $(m,C)$, since
$C$ is an operator. We develop the theory of the minimisation problem \eqref{e:problem1}
in terms of $\Gamma$ and extract results concerning the potential $v$ as
particular examples.

The end of Remark \ref{rem:32} shows that, without loss of generality, 
we can always restrict ourselves to covariance operators $C$ 
corresponding to Gaussian measures which are
equivalent to $\mu_0$.   In general the inverse $C^{-1}$ of such an operator and the inverse $\Cni$ of the covariance operator of $\mu_0$ do not have the same \emph{operator} domain. 
Indeed, see Example \ref{ex:delta} below for an example of two equivalent centred Gaussian measures whose inverse covariance operators have different domains. But item 1.) in the Feldman-Hajek Theorem (Proposition \ref{prop:FH}) implies that the domains of $C^{-\frac12}$ and $C^{-\frac12}_0$, i.e. the \emph{form domains} of $C^{-1}$ and $\Cni$, coincide. Hence, if we view the operators 
$C^{-1}$ and $\Cni$ as symmetric quadratic forms on $\cH^{1}$ or as operators from $\cH^1$ to $\cH^{-1}$ it makes sense to add and subtract them. In particular, we can interpret \eqref{e:defGamma} as
\begin{equation}\label{e:covariance}
 \Gamma:= C^{-1} - \Cni \in \LHo.
\end{equation}
Actually, $\Gamma$ is not only bounded from $\cH^1$ to $\cH^{-1}$.  Item 3.) in Proposition \ref{prop:FH} can be restated as  
\begin{equation}\label{e:FHH}
\big\| \Gamma  \big\|_{\mathcal{HS}(\cH^{1},\cH^{-1})}^2 :=  \big\| C_0^{\frac12} \Gamma C_0^{\frac12} \big\|_{\HS}^2     < \infty;
\end{equation}
here $\mathcal{HS}(\cH^{1},\cH^{-1})$ denotes the space of Hilbert-Schmidt 
operators from $\cH^1$ to $\cH^{-1}$.
The space $\mathcal{HS}(\cH^{1},\cH^{-1})$ is continuously
embedded into $ \LHo.$

Conversely, it is natural to ask if  condition \eqref{e:FHH} alone implies that $\Gamma$ can be obtained from the covariance of a Gaussian measure as in \eqref{e:covariance}. The following Lemma states that this is indeed the case as soon as one has an additional positivity condition; the proof is
left to the appendix.

\begin{lemma}\label{le:Gamma}
For any  symmetric $\Gamma$ in $\mathcal{HS}(\cH^{1},\cH^{-1})$ the quadratic form given by 
\begin{equation*}
Q_\Gamma(u,v) = \la u,\Cni v \ra + \la u, \Gamma v \ra,
\end{equation*}
is bounded from below and closed on its form domain $\cH^1$. Hence it is associated to a unique self-adjoint operator which we will also denote by $\Cni + \Gamma$. The operator $(\Cni + \Gamma)^{-1}$ is the covariance operator of a Gaussian measure on $\cH$ which is equivalent to $\mu_0$ if and only if $Q_\Gamma$ is  strictly positive.
\end{lemma}

Lemma \ref{le:Gamma} shows  that
we can parametrize the set of Gaussian measures that are equivalent to $\mu_0$ by their mean and by the operator $\Gamma$. For fixed $m \in \cH$ and $\Gamma \in \mathcal{HS}(\cH^{1},\cH^{-1})$ 
we write $\cNp(m,\Gamma)$ for the Gaussian measure with mean $m$ and covariance operator given by  $C^{-1}= \Cni+\Gamma$, where the suffix $(P,0)$ is to denote the specifiction  via the 
shift in precision operator from that of $\mu_0.$
We use the convention to set $\cNp(m,\Gamma) = \delta_m$ if $\Cni+\Gamma$ 
fails to be positive. Then we set
\begin{equation}\label{e:defA}
\cA := \{\nmG \in \Me \colon m \in \cH, \, \Gamma \in \mathcal{HS}(\cH^{1},\cH^{-1}) \}.
\end{equation}

Lemma \ref{le:Gamma} shows that the subset of $\cA$ in which $Q_\Gamma$ is
stricly positive comprises Gaussians measures absolutely continuous
with respect to $\mu_0.$ Theorem \ref{thm:Gaussian}, with the choice
$\cA=\cA_2$,  implies immediately the existence of a minimiser for problem \eqref{e:problem1} for this choice of $\cA$:

\begin{corollary}\label{cor:Gaussian}
Let $\mu_0$ be a Gaussian measure with mean $m_0 \in \cH$ and covariance operator $C_0 \in \TC$ and let $\mu$ be given by \eqref{e:target}. 
Consider $\cA$ given by
\eqref{e:defA}. 
Provided there exists a single $\nu \in \cA$ with $\Dnm < \infty$ 
then there exists a minimiser of $\nu \mapsto \Dnm$ in $\cA$. 
Furthermore, $\nu$ is necessarily equivalent to $\mu_0$ in the sense of measures.
\end{corollary}

However this corollary does not tell us much about the manner in
which minimising sequences approach the limit.
With some more work we can actually characterize the convergence more
precisely in terms of the parameterisation:

\begin{theorem}\label{thm:Gammaconv}
Let $\mu_0$ be a Gaussian measure with mean $m_0 \in \cH$ and covariance operator $C_0 \in \TC$ and let $\mu$ be given by \eqref{e:target}. 
Consider $\cA$ given by \eqref{e:defA}.
Let $\cNp(m_n,\Gn)$ be a sequence of Gaussian measures in $\cA$
that converge weakly to $\ns$ with 
\begin{equation*}
\Dkl(\nu_n\| \mu) \to \Dkl(\ns \| \mu).
\end{equation*}
Then $\ns = \cNp(m_\star, \Gs)$ and  
\begin{equation*}
\|m_n-m_\star\|_{\cH^1}+\big\| \Gn - \Gs\big\|_{ \mathcal{HS}(\cH^{1},\cH^{-1})} \to 0.
\end{equation*}
\end{theorem}
\begin{proof}
Lemma  \ref{le:GaussianConv} shows that $\nu_\star$ is Gaussian
and Theorem \ref{thm:Gaussian} that in fact $\nu_\star=\cNp(m_\star, \Gs)$.
It follows from Lemma \ref{le:CS} that $\nu_n$ converges to $\nu_\star$
in total variation.  Lemma \ref{le:FH} which follows shows that
\begin{equation*}
\big\| \Cs^{\frac12} \big( C_n^{-1} - \Cs^{-1}\big) \Cs^{\frac12} \big\|_{\HS}
+\| m_n - m_\star \|_{\cH^1} \to 0.
\end{equation*} 
By Feldman-Hajek Theorem (Proposition \ref{prop:FH}, item 1.)) 
the Cameron-Martin spaces $\Cs^{\frac12}\cH$
and $C_0^{\frac12}\cH$ coincide with $\cH^1$ and hence, since
$C_n^{-1} - \Cs^{-1}=\Gamma_n-\Gamma_{\star}$, the desired result follows.
\end{proof}

The following example concerns a subset of
the set $\cA$ given by \eqref{e:defA} found by writing
$\Gamma$ a multiplication by a constant. This structure 
is useful for numerical computations, for example if
$\mu_0$ represents Wiener measure (possibly conditioned)
and we seek an approximation $\nu$ to $\mu$ with a 
mean $m$ and covariance of Ornstein-Uhlenbeck type (again possibly
conditioned).

\begin{example} 
Let $C^{-1}=C_0^{-1}+\beta I$ so that
\begin{equation}
C=(I+\beta C_0)^{-1}C_0.
\label{eq:Cb}
\end{equation}
Let $\cA'$ denote the set
of Gaussian measures on $\cH$ which 
have covariance of the form \eqref{eq:Cb} 
for some constant $\beta \in \R$. This set is parameterized by
the pair $(m,\beta) \in \cH \times \R$. Lemma \ref{le:Gamma} above states that 
$C$ is the covariance of a Gaussian equivalent to $\mu_0$ 
if and only if $\beta \in
\Is=(-\lambda_1^{-1},\infty)$; recall that $\lambda_1$, defined above \eqref{e:projectgamma} is the largest eigenvalue of $C_0$. 
Note also that the covariance $C$ satisfies $C^{-1}=C_0^{-1}+\beta$
and so $\cA'$ is a subset of $\cA$ given by \eqref{e:defA}
arising where $\Gamma$ is multiplication by a constant. 

Now consider minimising sequences $\{\nu_n\}$ from $\cA'$ for $\Dkl(\nu\|\mu).$
Any weak
limit $\ns$ of a sequence $\nu_n=N\bigl(m_n,(I+\beta_n C_0)^{-1}C_0\big) 
\in \cA'$ is necessarily Gaussian by Lemma \ref{le:GaussianConv}, 1.) 
and we denote it by $N(m_\star,\Cs).$ By 2.) of the same lemma we deduce
that $m_n \to m_\star$ strongly in $\cH$ and by 3.) that 
$(I+\beta_n C_0)^{-1}C_0 \to \Cs$ strongly in ${\cal L}(\cH).$
Thus, for any $\alpha\geq 1$, and recalling that $\ea$ are the eigenvectors
of $C_0$,
$\|\Cs\ea-(1+\beta_n\lambda_\alpha)^{-1}\lambda_\alpha \ea\| \to 0$
as $n \to \infty.$
Furthermore, necessarily $\beta_n \in \Is$ for each $n$. We now argue
by contradiction that there are no subsequences $\beta_{n'}$ converging
to either $-\lambda_1^{-1}$ or $\infty.$ For contradiction assume
first that there is a subsequence converging to $-\lambda_1^{-1}$.
Along this subsequence we have $(1+\beta_n\lambda_1)^{-1} \to \infty$
and hence we deduce that $\Cs e_1=\infty$, so that
$\Cs$ cannot be trace-class, a contradiction. 
Similarly assume for
contradiction that there is a subsequence converging to $\infty.$
Along this subsequence we have $(1+\beta_n\lambda_\alpha)^{-1} \to 0$
and hence that $\Cs \ea=0$ for every $\alpha$. In this case $\ns$ would be a 
Dirac measure, and hence not equivalent to  $\mu_0$ (recall our assumption that $C_0$ is a strictly 
positive operator). Thus there must be a subsequnce
converging to a limit $\beta \in \Is$ and we deduce that
$\Cs \ea=(1+\beta\lambda_\alpha)^{-1}\lambda_\alpha \ea$ proving
that $\Cs =(I+\beta C_0)^{-1}C_0$ as required. 
\end{example}

Another class of Gaussian which is natural in applications, and
in which the parameterization of the covariance is finite dimensional,
is as follows.

\begin{example}
Recall the notation $\pg$ for the orthogonal projection onto 
$\Hg := \spa(e_1, \ldots , e_\gamma)$ the span of the first
$\gamma$ eigenvalues of $C_0.$ We seek $C$ in the form
$$C^{-1}=\bigl((I-\pg)C_0(I-\pg)\bigr)^{-1}+\Gamma$$
where
$$\Gamma=\sum_{i,j \le N} \gamma_{ij} e_i \otimes e_j.$$
It then follows that
\begin{equation}
C=(I-\pg)C_0(I-\pg)+\Gamma^{-1},
\label{eq:Cbb}
\end{equation}
provided that $\Gamma$ is invertible.
Let $\cA'$ denote the set
of Gaussian measures on $\cH$ which 
have covariance of the form \eqref{eq:Cbb} 
for some operator $\Gamma$ invertible on $\Hg.$
Now consider minimising sequences $\{\nu_n\}$ from $\cA'$ for $\Dkl(\nu\|\mu)$
with mean $m_n$ and covariance $C_n=(I-\pg)C_0(I-\pg)+\Gamma_n^{-1}.$
Any weak limit $\ns$ of the sequence $\nu_n \in \cA'$ is necessarily Gaussian by Lemma \ref{le:GaussianConv}, 1.) 
and we denote it by $N(m_\star,\Cs).$ As in the preceding example,
we deduce that $m_n \to m_\star$ strongly in $\cH$. Similarly we also
deduce that $\Gamma_n^{-1}$  converges to a non-negative matrix.
A simple contradiction shows that, in fact, 
this limiting matrix is invertible since
otherwise $N(m_\star,\Cs)$  would not be equivalent to $\mu_0$. We denote 
the limit by $\Gamma_{\star}^{-1}$. We deduce that the limit
of the sequence $\nu_n$ is in $\cA'$ and that 
$\Cs=(I-\pg)C_0(I-\pg)+\Gamma_{\star}^{-1}.$
\end{example}

\subsection{Regularisation for Parameterisation of Gaussian Measures}
\label{ssec:REG}

The previous section demonstrates that parameterisation of Gaussian
measures in the set $\cA$ given by \eqref{e:defA} leads to 
a well-defined minimisation problem \eqref{e:problem1} and that,
furthermore, minimising sequences in $\cA$ will give rise to
means $m_n$ and operators $\Gamma_n$ converging in $\cH^1$
and $ \mathcal{HS}(\cH^{1},\cH^{-1})$ respectively. However,
convergence in the space $ \mathcal{HS}(\cH^{1},\cH^{-1})$ 
may be quite weak and unsuitable for numerical purposes;
in particular if $\Gamma_n u=v_n(\cdot)u(\cdot)$ then the
sequence $(v_n)$ may behave quite badly, even though $(\Gamma_n)$ is
well-behaved in $ \mathcal{HS}(\cH^{1},\cH^{-1})$. 
For this reason we consider, in this subsection,
regularisation of the minimisation problem
\eqref{e:problem1} over $\cA$ given by \eqref{e:defA}. But before doing so
we provide two examples illustrating the potentially 
undesirable properties of convergence
in $ \mathcal{HS}(\cH^{1},\cH^{-1})$.
\def\RS{ \cite[Example 3 in Section X.2]{RS2}}

\begin{example}[Compare \RS ]\label{ex:delta}
Let $\Cni = - \partial_t^2$ be the negative Dirichlet-Laplace operator on $[-1,1]$ 
with domain $H^2([-1,1]) \cap H^1_0([-1,1]),$   
and let $\mu_0= \cN(0,C_0)$, i.e. $\mu_0$ is the distribution of a Brownian bridge on $[-1,1]$. In this case $\cH^1$ coincides with the Sobolev space $H^1_0$. We note that the measure $\mu_0$ assigns  full mass to the space 
$X$ of continuous functions on $[-1,1]$ and 
hence all integrals with respect to
$\mu_0$ in what follows can be computed over $X$.
Furthermore, the centred unit ball in $X$, 
\begin{equation*}
B_{X}(0;1) := \Big\{x \in X\colon \sup_{t\in [-1,1]} |x(t)| \leq 1\Big\},
\end{equation*}
has
positive $\mu_0$ measure.

Let $\ph\colon \R \to \R$ be a standard mollifier, i.e. $\ph \in \cC^\infty$, 
$\phi \geq 0$, $\ph$ is compactly supported in $[-1,1]$ and $\int_\R \ph(t) \, dt =1$. Then for any $n$ define $\ph_n(t) = n \ph(t n)$, 
together with the probability measures $\nu_n \ll \mu_0$ given by by 
\begin{equation*}
\frac{d \nu_n}{ d \mu_0} (x(\cdot) ) =\frac{1}{Z_n} \exp \Big(- \frac12 \int_{-1}^1 \ph_n(t) \, x(t)^2 \, dt  \Big),
\end{equation*} 
where
\begin{equation*}
Z_n:= \E^{\mu_0}  \exp \Big(- \frac12 \int_{-1}^1 \ph_n(t) \, x(t)^2 \, dt  \Big).
\end{equation*}
The $\nu_n$ are also Gaussian, as Lemma \ref{le:PI} shows.
Using the fact that $\mu_0(X)=1$ it follows that 
\begin{equation*}
\exp(-1/2)\mu_0\bigl(B_{X}(0;1)\bigr) \le Z_n \le 1.
\end{equation*}
Now define probability measure $\nu_{\star}$ by
\begin{equation*}
\frac{d \ns}{ d \mu_0} (x(\cdot) ) = \frac{1}{Z_\star} \exp \bigg(- \frac{ x(0)^2}{2}  \bigg)   
\end{equation*}
noting that
\begin{equation*}
\exp(-1/2)\mu_0\bigl(B_{X}(0;1)\bigr) \le Z_{\star} \le 1.
\end{equation*}
For any  $x \in X$ we have
\begin{equation*}
\int_{-1}^1 \ph_n(t) \, x(t)^2 \, dt \to  x(0)^2 .
\end{equation*}
An application of the dominated convergence theorem shows that
$Z_n \to Z_\star$ and hence that $Z_n^{-1} \to Z_\star^{-1}$ and 
$\log(Z_n) \to \log(Z_\star).$

Further applications of the dominated convergence theorem show
that the $\nu_n$ converge weakly to $\ns$, which is also then Gaussian
by Lemma \ref{le:GaussianConv}, and that the
the Kullback-Leibler divergence between $\nu_n$ and $\nu_{\star}$
satisfies 
\begin{align}
\Dkl(\nu_n \|\ns) &=\frac{1}{Z_n} \E^{\mu_0} \Big[ \exp \Big(- \frac12 \int_{-1}^1 \ph_n(t) \, x(t)^2 \, dt  \Big)\notag \\
&   \times \frac12\Big(x(0)^2- \int_{-1}^1 \ph_n(t) \, x(t)^2 \, dt\Big) \Big] +\big( \log(Z_\star) - \log(Z_n)\big)
& \to 0.\notag
\end{align}
Lemma \ref{le:PI} shows that $\nu_n$ is the centred Gaussian
with covariance $C_n$ given by $C_n^{-1}=C_0^{-1}+\ph_n.$
Formally, the covariance operator associated to $\ns$ is given by $\Cni  + \delta_0$, where $\delta_0$ is the Dirac $\delta$ function. 
Nonetheless
the implied mutiplication operators converge to a limit in 
$ \mathcal{HS}(\cH^{1},\cH^{-1})$. 
In applications such limiting behaviour of the potential in an inverse
covariance representation, to a distribution, may be computationally
undesirable.
\end{example}

\begin{example}\label{ex:averaging}
We consider a second example in a similar vein, but linked to the theory of
averaging for differential operators. Choose $\mu_0$ as in
the preceding example and now
define $\vvarphi_n(\cdot)=\vvarphi(n\cdot)$ where $\vvarphi:\R \to \R$ 
is a positive smooth $1-$periodic function with mean $\bvp$.
Define $C_n$ by $C_n^{-1} = \Cni + \pph_n$ similarly to before. It follows, as in the previous example, by use of Lemma \ref{le:PI},
 that the measures $\nu_n$ are centred Gaussian with covariance $C_n$,
are equivalent to $\mu_0$ and
\begin{equation*}
\frac{d \nu_n}{ d \mu_0} (x(\cdot) ) =\frac{1}{Z_n}\exp \Big(- \frac12 \int_{-1}^1 \pph_n(t) \, x(t)^2 \, dt  \Big). 
\end{equation*} 
By the dominated convergence theorem, as in the previous example,
it also follows that the $\nu_n$ converge weakly to $\ns$ with 
\begin{equation*}
\frac{d \ns}{ d \mu_0} (x(\cdot) ) = \frac{1}{Z_\star}\exp \bigg(- \frac12 \bvp \int_{-1}^1  x(t)^2\,dt  \bigg)  . 
\end{equation*}
Again using Lemma \ref{le:PI}, $\ns$ is the centred Gaussian with
covariance $C_{\star}$ given by $C_{\star}^{-1}=C_0^{-1}+\bvp.$
The Kullback-Leibler divergence satisfies 
$\Dkl(\nu_n \|\ns) \to 0$,
also by application of the dominated convergence theorem as in the
previous example.
Thus minimizing sequences may exhibit multiplication functions which oscillate
with increasing frequency whilst the implied operators $\Gn$ converge
in $\mathcal{HS}(\cH^{1},\cH^{-1})$. 
Again this may be computationally undesirable in many applications.
\end{example}

The previous examples suggest that, in order to induce improved behaviour of 
minimising sequences related to the the operators $\Gamma$, in
particular when $\Gamma$ is a mutiplication operator, it 
may be useful to regularise the minimisation in problem \eqref{e:problem1}. 
To this end, let $\cG \subseteq \mathcal{HS}(\cH^{1},\cH^{-1})$ be a Hilbert space of  linear operators. For fixed $m \in \cH$ and $\Gamma \in \cG$ we write $\nmG$ for the Gaussian measure with mean $m$ and covariance operator given by \eqref{e:covariance}. We now make the choice
\begin{equation}
\cA := \{\nmG \in \Me \colon m \in \cH, \, \Gamma \in \cG \}.
\label{eq:choice}
\end{equation}
Again, we use the convention $\nmG = \delta_0$ if $\Cni+\Gamma$ fails to be positive. Then, for some $\delta >0$ we consider the modified minimisation problem
\begin{equation}
\underset{\nu \in \cA} \argmin \Bigl(\Dkl(\nu,\mu)  + \delta \| \Gamma \|_{\cG}^2\Bigr).  \label{e:problem2}
\end{equation}

We have existence of minimisers for problem \eqref{e:problem2} under very general assumptions. In order to state these assumptions, we introduce auxiliary interpolation spaces. For any $s >0$, we denote by  $\cH^s$ the domain of $C_0^{-\frac{s}{2}}$ equipped with the scalar product $\la \cdot, C_0^{-s} \cdot \ra$ and define $\cH^{-s}$ by duality.

\begin{theorem}\label{thm:existence2}
Let $\mu_0$ be a Gaussian measure with mean $m_0 \in \cH$ and covariance operator $C_0 \in \TC$ and let $\mu$ be given by \eqref{e:target}. 
Consider $\cA$ given by \eqref{eq:choice}.
Suppose that the space $\cG$ consists of symmetric operators on $\cH$ and embeds compactly into the space of bounded linear operators from $\cH^{1-\kappa}$ to $\cH^{-(1-\kappa)}$ for some $0<\kappa <1$. Then, provided that
$\Dkl( \mu_0 \| \mu) < \infty$, there exists a minimiser 
$\ns=\nmsG$ for problem \eqref{e:problem2}. 

Furthermore, along any minimising sequence $\nu(m_n,\Gamma_n)$  there is a subsequence $\nu(m_{n'}, \Gamma_{n'})$ along which
$\Gamma_{n'} \to \Gamma_\star$ strongly in $\cG$ and $\nu(m_{n'}, \Gamma_{n'}) \to\nu(m_\star,\Gs)$ with respect to the total variation distance.
\end{theorem}

\begin{proof}
The assumption $\Dkl(\mu_0\| \mu) < \infty$ implies that the infimum in \eqref{e:problem2} is finite and non-negative.
Let $\nu_n = \cNp(m_n,\Gn)$ be a minimising sequence for \eqref{e:problem2}. As both $ \Dkl(\nu_n\| \mu)$  and  $ \| \Gn \|_{\cG}^2$ are non-negative this implies that $\Dkl(\nu_n\| \mu)$ and  $ \| \Gn \|_{\cG}^2$  are bounded along the sequence. Hence, by Proposition \ref{le:firstproperties} and by the compactness assumption on $\cG$,  after passing to a subsequence twice we can assume that the measures $\nu_n $ converge weakly as probability measures to a measure $\ns$ and the operators $\Gn$ converge weakly in $\cG$ to an operator $\Gs$; furthermore the $\Gn$ 
also converge in the operator norm of $\LHk$ to $\Gs$.   By lower semicontinuity of $\nu \mapsto \Dnm$ with respect to weak convergence of probability measures (see Proposition \ref{le:firstproperties}) and by lower semicontinuity of $\Gamma \mapsto \| \Gamma\|_{\cG}^2$ with respect to weak convergence in  $\cG$ we can conclude that 
\begin{align}
\Dkl(\ns \| \mu) + \delta \| \Gs \|_{\cG}^2 &\leq
\liminf_{n \to \infty}\Dkl(\nu_n\|\mu)+\liminf_{n \to \infty}
\delta\|\Gn\|_{\cG}^2 \notag\\
&\leq \lim_{n \to \infty}\Bigl(\Dkl(\nu_n\|\mu)+\delta\|\Gn\|_{\cG}^2 \Bigr)\notag\\
& =\underset{\nu \in \cA} \inf \Bigl(\Dkl(\nu\|\mu) + \delta \| \Gamma \|_{\cG}^2\Bigr). \label{e:regmin1}
\end{align}

By Lemma \ref{le:GaussianConv} $\ns$ is a Gaussian measure with mean $m_\star$ and covariance operator $\Cs$  and we have 
\begin{equation}\label{e:proofthm1}
\| m_n - m_\star \|_{\cH} \to 0 \qquad \text{and} \qquad \| C_n -C_\star \|_{\cL(\cH)} \to 0.
\end{equation}
We want to show that $C_\star = (C_0 + \Gs)^{-1}$ in the sense of Lemma \ref{le:Gamma}. In order to see this, note that $\Gs \in \LHk$ which implies that for $x \in \cH^1$ we have for any $\lambda >0$
\begin{align}
\la  x, \Gs x\ra &\leq \big\| \Gs \big\|_{\LHk}\big\| x \|_{\cH^{1-\kappa}}^2\notag\\
 &\leq \big\| \Gs \big\|_{\LHk}   \Big(\lambda (1-\kappa) \big\| x \|_{\cH^{1}}^2 + \lambda^{-\frac{1-\kappa}{\kappa}}\kappa \big\| x \|_{\cH}^2 \Big)\notag.
\end{align}
Hence, $\Gs$ is infinitesimally form-bounded with respect to $\Cni$ (see e.g. \cite[Chapter X.2]{RS2}). In particular, by the KLMN theorem (see \cite[Theorem X.17]{RS2}) the form $\la x, \Cni x \ra + \la x,\Gs x\ra$ is bounded from below and closed. Hence there exists a unique self-adjoint operator denoted by $\Cni + \Gs$ with form domain $\cH^1$ which generates this form. 

The convergence of $C_n = (\Cni + \Gn)^{-1}$ to $\Cs$ in $\cL(\cH)$ implies in particular, that the $C_n$ are bounded in the operator norm, and hence the spectra of the $\Cni +\Gn$ are away from zero from below, uniformly. 
This implies that 
\begin{equation*}
\inf_{\|x\|_{ \cH}=1} \Bigl(\la x, \Cni x \ra + \la x, \Gs x\ra\Bigr) \geq \liminf_{n \to \infty} \inf_{\|x\|_{ \cH}=1} \Bigl(\la x, \Cni x \ra + \la x, \Gn x\ra \Bigr)> 0,
\end{equation*}
so that $\Cni +\Gs$ is a positive operator and in particular invertible and so is $(\Cni +\Gs)^{\frac12}$.  As $(\Cni +\Gs)^{\frac12}$ is defined on all of $\cH^1$ its inverse maps onto $\cH^1$. Hence, the closed graph theorem implies that $C_0^{-\frac12}(\Cni +\Gs)^{-\frac12} $ is a bounded operator on $\cH$. From this we can conclude that for all $x \in \cH^1$
\begin{align*}
\big| \la&  x, (\Cni +\Gn) x \ra -  \la x, (\Cni +\Gamma_{\star}) x \ra  \big|\\
 &\leq \big\| \Gn - \Gs \big\|_{\LHo} \big\| x\big\|_{\cH^1}^2\\
&\leq \big\| \Gn - \Gs \big\|_{\LHo} \big\| C_0^{-\frac12}  (\Cni +\Gs)^{-\frac12} \big\|_{\cL(\cH)}^2  \big\| (\Cni + \Gs \big)^{\frac12} x\big\|_{\cH}^2.
\end{align*}
By \cite[Theorem VIII.25]{RS1} this implies that $\Cni +\Gs$ converges to $\Cni + \Gs$ in the strong resolvent sense. As all operators are positive and bounded away from zero by \cite[Theorem VIII.23]{RS1} we can conclude  that the inverses $(\Cni +\Gn)^{-1}$ converge to $(\Cni + \Gs)^{-1}$. By \eqref{e:proofthm1} this implies that $C_\star = (\Cni + \Gs)^{-1}$ as desired. 

We can conclude that $\ns=\nmsG$ and hence that 
\begin{equation*}
\Dkl(\ns \| \mu) + \delta \| \Gs \|_{\cG}^2 \geq \underset{\nu \in \cA} \inf \Bigl(\Dkl(\nu\|\mu) + \delta \| \Gamma \|_{\cG}^2\Bigr),
\end{equation*}
implying from \eqref{e:regmin1} that
\begin{align*}
\Dkl(\ns \| \mu) + \delta \| \Gs \|_{\cG}^2 &=
\liminf_{n \to \infty}\Dkl(\nu_n\|\mu)+\liminf_{n \to \infty}
\delta\|\Gn\|_{\cG}^2 \notag\\
&= \lim_{n \to \infty}\Bigl(\Dkl(\nu_n\|\mu)+\delta\|\Gn\|_{\cG}^2 \Bigr)\notag\\
& =\underset{\nu \in \cA} \inf \Bigl(\Dkl(\nu\|\mu) + \delta \| \Gamma \|_{\cG}^2\Bigr). 
\end{align*}
Hence we can deduce using the lower semi-continuity of $\Gamma \mapsto \|\Gamma \|_{\cG}^2$ with respect to weak convergence in $\cG$ 
\begin{align*}
\limsup_{n \to \infty} \Dkl(\nu_n \| \mu) 
&\leq \lim_{n \to \infty} \Bigl(\Dkl(\nu_n\|\mu)+\delta\|\Gn\|_{\cG}^2 \Bigr) - \liminf_{n \to \infty} \delta\|\Gn\|_{\cG}^2\\
&\leq \Bigl(\Dkl(\nu_{\star}\|\mu)+\delta\|\Gamma_{\star}\|_{\cG}^2 \Bigr) -  \delta\|\Gamma_{\star}\|_{\cG}^2\\
&= \Dkl(\ns \| \mu) ,
\end{align*}
which implies that $\lim_{n \to \infty} \Dkl(\nu_n \| \mu) = \Dkl(\ns \| \mu)$. In the same way it follows that $\lim_{n \to \infty } \|\Gamma_n\|_{\cG}^2 = \|\Gs \|_{\cG}^2$.
By Lemma \ref{le:CS} we can conclude that $\|\nu_n - \ns \rTV \to 0$.  For the operators $\Gn$ we note that weak convergence together with convergence of the norm implies strong convergence.

\end{proof}

\vskip2ex

\begin{example}
The first example we have in mind is the case where, 
as in Example \ref{ex:delta},  $\cH=L^2([-1,1])$,  $\Cni$ is the negative Dirichlet-Laplace operator on $[-1,1]$,  $\cH^1= H^1_0$, and $m_0=0$. Thus
the reference measure is the distribution of a centred Brownian bridge. 
By a slight adaptation of the proof of \cite[Theorem 6.16]{martinSPDE})
we have that, for $p \in (2,\infty]$, $\|u\|_{L^p} \le C\|u\|_{\cH^s}$ 
for all $s>\frac{1}{2}-\frac{1}{p}$ and we will use this fact in what follows.
For $\Gamma$ we chose multiplication operators with  suitable functions  $\hat{\Gamma}\colon [-1,1] \to \R$. For any $r>0$ we denote by $\cG^r$ the space of multiplication operators with functions $\hat{\Gamma} \in  H^r([-1,1])$ endowed with the Hilbert space structure of $H^r([-1,1])$. In this notation, the compact embedding of the spaces $H^r([-1,1])$ into $L^2([-1,1])$, can be rephrased as a compact embedding of the  space $\cG^r$  into the space $\cG^0$, i.e. the space of $L^2([-1,1])$ functions, viewed as multiplication operators. By the form of Sobolev 
embedding stated above we have that for $\kappa<\frac34$ and any
\footnote{
Throughout the paper we write $a \lesssim b$ to indicate that there exists a constant $c>0$ independent of the relevant quantities such that $a \leq c b$.}
 $x \in \cH^{1-\kappa}$
\begin{equation}\label{e:imbed}
\la x, \Gamma x \ra = \int_{-1}^{1} \hat{\Gamma}(t)  x(t)^2  dt  \leq  \| \hat{\Gamma} \|_{L^2([-1,1])} \|x\|_{L^4}^2 \lesssim  \| \hat{\Gamma} \|_{L^2([-1,1])} \|x\|_{\cH^{1-\kappa}}^2.
\end{equation}
Since this shows that
$$\|\Gamma\|_{\cL(\cH^{1-\kappa},\cH^{-(1-\kappa)})} \lesssim \| \hat{\Gamma} \|_{L^2([-1,1])}$$
it demonstrates that $\cG_0$ embeds continuously into the space $\LHk$ and hence, the spaces $\cG^r$, which are compact in $\cG_0$,
satisfy the assumption of Theorem \ref{thm:existence2} for any $r>0$. 
\end{example}

\begin{example}
Now consider $\mu_0$ to be a Gaussian field over a space of dimension $2$
or more. In this case we need to take a covariance operator that has a stronger regularising property than the inverse Laplace operator. For example, if we denote by $\Delta$ the Laplace operator on the $n$-dimensional torus $\Tn$, then the Gaussian field with covariance operator $C_0  = (-\Delta + I)^{-s}$ takes values in $L^2(\Tn)$ if and only if $ s> \frac{n}{2}$. In this case, the space $\cH^1$ coincides with the fractional Sobolev space $H^{s}(\Tn)$. Note that the condition $s>\frac{n}{2}$ precisely implies that there exists a $\kappa>0$ such that the space $\cH^{1-\kappa}$ embeds into $L^{\infty}(\Tn)$ and in particular into $L^4[0,T]$.  As above, denote by $\cG^r$ the space of multiplication operators on $L^2(\Tn)$ with functions $\hat{\Gamma} \in H^r(\Tn)$. Then the same calculation as \eqref{e:imbed} shows that the conditions of Theorem \ref{thm:existence2} are satisfied for any $r>0$.

\end{example}

\subsection{Uniqueness of Minimisers}
\label{ssec:uniqueness}

As stated above in Proposition \ref{thm:CS}, the minimisation
problem \eqref{e:problem1} has a unique minimiser if the set $\cA$ is convex. 
Unfortunately, in all of the situations discussed in this section,
$\cA$ is not convex, and in general we cannot expect  minimisers to be unique;
the example in subsection~\ref{ssec:Examples} illustrates nonuniqueness.
There is however one situation in which we have uniqueness for all of the  choices of $\cA$ discussed in Theorem \ref{thm:Gaussian}, namely the case of where instead of $\cA$ the measure $\mu$ satisfies a convexity property. 
Let us first recall  the definition  of $\lambda$-convexity. 

\begin{definition}
Let $\Phi \colon \cH^1 \to \R$ be function. For a $\lambda \in \R$ the function $\Phi$ is  \emph{$\lambda$-convex} with respect to $\cH^1$ if 
\begin{equation}\label{e:laco}
\cH^1  \ni x \mapsto  \frac{\lambda}{2}  \la x,x\ra_{\cH^1} + \Phi(x)
\end{equation}
is convex on $\cH^1$.
\end{definition}
\begin{remark}
Equation \eqref{e:laco} implies that for any $x_1,x_2 \in \cH^1$ and for any $t \in (0,1)$ we have
\begin{equation}\label{e:laco2}
\Phi((1-t) x_1 + t x_2 ) \leq (1-t) \Phi(x_1) + t \Phi(x_2) + \lambda \frac{t(1-t)}{2}  \| x_1 -x_2 \|_{\cH^1}^2.
\end{equation}
Equation \eqref{e:laco2} is often taken to define $\lambda$-convexity because it gives useful estimates even when the distance function does not come from a scalar product. For Hilbert spaces both definitions are equivalent.
\end{remark}


The following theorem implies uniqueness for 
the minimisation problem \eqref{e:problem1}
as soon as $\Phi$ is $(1-\kappa)$-convex for a $\kappa >0$ and satisfies a mild integrability property.
The proof is given in section \ref{sec:proofs}.

\begin{theorem}\label{thm:uniqueness}
Let $\mu$ be as in \eqref{e:target} and assume that there exists a $\kappa > 0$ such that $\Phi$ is $(1-\kappa)$-convex with respect to $\cH^1$.  Assume that there exist constants $0<c_i< \infty$, $i=1,2,3$, and $\alpha \in (0,2)$ such that for every $x \in X$ we have
\begin{equation}\label{e:PhiBo}
-c_1 \|x \|_X^\alpha \leq  \Phi(x) \leq c_2 \exp\big(c_3 \|x\|_X^\alpha \big).
\end{equation}
Let $\nu_1 = \cN(m_1,C_1)$ and  $\nu_2= \cN(m_2, C_2)$  be Gaussian measures with $\Dkl(\nu_1 \| \mu)< \infty$ and $\Dkl(\nu_2 \| \mu)< \infty$. For any $t \in (0,1)$ there exists an interpolated measure $\nt= \cN(m_t, C_t)$ which satisfies  $\Dkl(\nt \| \mu)< \infty$. 
Furthermore, as soon as $\nu_1 \neq \nu_2$ there exists a constant $K>0$ such that for all $t \in (0,1)$ 
\begin{equation*}
\Dkl(\nt \| \mu) \leq  (1-t)  \Dkl(\nu_1 \| \mu) +t  \Dkl(\nu_1 \| \mu) - \frac{t(1-t)}{2}K.
\end{equation*}
Finally, if we have  $m_1=m_2$ then $m_t = m_1$ holds as well for all $t \in (0,1)$, and in the same way, if $C_1 = C_2$, then $C_t = C_1$ for all $t \in (0,1)$.

\end{theorem}

The measures $\nt$ introduced in Theorem \ref{thm:uniqueness} are a special case of geodesics on Wasserstein space first introduced in \cite{McCann} in a finite dimensional situation. In addition, the proof shows that the constant $K$  
appearing in the statement is $\kappa$ times the square of the Wasserstein distance between $\nu_1$ and $\nu_2$ with respect to the $\cH^1$ norm. 
 See \cite{AGS,FU} for a more detailed discussion of mass transportation on infinite dimensional spaces. The following is an immediate consequence of
Theorem \ref{thm:uniqueness}: 

\begin{corollary}
Assume that $\mu$ is a probability measure given by
\eqref{e:target}, that  there exists a $\kappa >0$ such that $\Phi$ is $(1-\kappa)$ convex with respect to $\cH^1$ and that $\Phi$ satisfies the bound \eqref{e:PhiBo}. 
Then for any of the four choices of sets $\cA_i$ discussed in Theorem \ref{thm:Gaussian} the minimiser of $\nu \mapsto \Dnm$ is unique in $\cA_i$.
\end{corollary}
\begin{remark}
The assumption that $\Phi$ is $(1-\kappa)$-convex for a $\kappa >0$ implies in particular that $\mu$ is \emph{log-concave} (see \cite[Definition 9.4.9]{AGS}). It can be viewed as a quantification of this log-concavity.
\end{remark}
\begin{example}
As in Examples \ref{ex:delta} and \ref{ex:averaging} above, let  $\mu_0$ be a centred Brownian bridge on $[-\frac{L}{2},\frac{L}{2}]$. As above we have $\cH^1 = H^1_0([-\frac{L}{2},\frac{L}{2}])$ equipped with the \emph{homogeneous} Sobolev norm and $X = C([-\frac{L}{2}, \frac{L}{2}])$. 

For some $\cC^2$ function $\ph\colon \R \to \R_+$  set  $\Phi\big(x(\cdot)\big) = \int_{-\frac{L}{2}}^{\frac{L}{2}} \ph(x(s)) \, ds$. The  integrability condition \eqref{e:PhiBo} translates immediately into the growth condition $- c_1' |x|^\alpha \leq \ph(x) \leq c_2' \exp(c_3' |x|^{\alpha})$ for $x \in \R$ and constants $0< c_i' < \infty$ for $i = 1,2,3$. 
Of course, the convexity assumption of Theorem \ref{thm:uniqueness} is satisfied if $\ph$ is convex. But  we can allow for some non-convexity. For example, if $\ph \in \cC^2(\R)$ and $\ph''$ is uniformly bounded from below by $-K \in \R$, then we get for $x_1,x_2 \in \cH^1$
\begin{align}
\Phi((1-t)& x_1 + t x_2 \big) \notag\\
= &\int_{-\frac{L}{2}}^{\frac{L}{2}} \ph\big( (1-t) x_1(s) + t x_2(s) \big)   \, ds  \notag\\
\leq &\int_{-\frac{L}{2}}^{\frac{L}{2}}  (1-t) \ph\big( ( x_1(s)  \big) + t \ph\big(x_2(s) \big)   + \frac{1}{2} t(1-t) K\big|  x_1(s) -x_2(s) \big|^2\, ds \notag \\
= & (1-t) \Phi(x_1) + t \Phi(x_2) + \frac{Kt(1-t)}{2} \int_{-\frac{L}{2}}^{\frac{L}{2}}  \big|  x_1(s) -x_2(s) \big|^2 ds . \notag
\end{align}
Using the  estimate 
\begin{align*}
 \int_{-\frac{L}{2}}^{\frac{L}{2}} \big|  x_1(s) -x_2(s) \big|^2 ds&     \leq\bigg( \frac{L}{\pi}\bigg)^2 \|x_1 -x_2 \|_{\cH^1}^2
\end{align*}
we see that $\Phi$ satisfies the convexity assumption as soon as $K < \big(\frac{\pi}{L} \big)^2$.
\end{example}
The proof of Theorem \ref{thm:uniqueness} is based on the influential concept of displacement convexity, introduced by McCann in  \cite{McCann}, and heavily inspired by the infinite dimensional exposition in \cite{AGS}. It can be found in Section~\ref{sec:proofUniqueness}.

\subsection{Gaussian Mixtures}
\label{ssec:mixtures}

We have demonstrated a methodology for approximating measure $\mu$ given
by \eqref{e:target} by a Gaussian $\nu$. If $\mu$ is multi-modal
then this approximation can result in several local minimisers
centred on the different modes. A potential way to capture all modes
at once is to use Gaussian mixtures, as explained in the finite
dimensional setting in \cite{bishop2006pattern}. We explore this
possibility in our infinite dimensional context:
in this subsection we show existence of minimisers for problem 
\eqref{e:problem1} in the situation when we are minimising over a set of convex combinations of Gaussian measures.

We start with a baisc lemma for which
we do not need to assume that the mixture measure comprises Gaussians.
\begin{lemma}\label{le:convexcombination}
Let $\cA, \cB \subseteq \Me$ be closed under weak convergence of probability measures. Then so is
\begin{equation*}
\cC := \big\{  \mu := p^1 \nu^1+ p^2\nu^2 \colon 0 \leq p^{i} \leq 1, i=1,2; \quad  p^1 + p^2=1; \,\nu^1 \in \cA; \, \nu^2 \in \cB   \}
\end{equation*}
\end{lemma}
\begin{proof}
Let $(\nu_n)= (p^1_n \nu^1_n + p^2_n \nu^2_n )$ be a sequence of measures in $\cC$ that converges weakly to $\mu_\star \in \Me$. We want to show that $\mu_\star \in \cC$. It suffices to show that a subsequence of the $\nu_n$ converges to an element in $\cC$.
After passing to a subsequence we can assume that for $i=1,2$ the $p^{i}_n$ converge to $p^{i}_\star \in [0,1]$ with $p^1_\star + p^2_\star=1$. 
Let us first treat the case where one of these $p^{i}_\star$ is zero -- say $p^1_\star=0$ and $p^2_\star=1$. In this situation we can conclude that the $\nu^2_n$ converge weakly to $\mu_\star$ and hence $\mu_\star \in \cB \subseteq \cC$.
Therefore, we can  assume  $p^{i}_\star \in (0,1)$. After passing to another subsequence we can furthermore assume that the $p^{i}_n$ are uniformly bounded from below by a positive constant $\hat{p} >0$. As the sequence $\nu_n$ converges weakly in $\Me$ it  is  tight. We claim that this implies automatically the tightness of the sequences $\nu_n^{i}$. Indeed, for a $\delta >0$ let $K_\delta \subseteq \cH$ be a compact set with $\nu_n(K_\delta) \leq \delta$ for any $n \geq 1$. Then we have for any $n$ and for $i=1,2$  that 
\begin{equation*}
\nu^{i}_n(K_\delta) \leq \frac{1}{\hat{p}} \nu(K_\delta) \leq \frac{\delta}{\hat{p}}.
\end{equation*}
After passing to yet another subsequence, we can assume that the $\nu^{1}_n$ converge weakly to $\nu^{1}_\star \in \cA$ and the $\nu^{2}_n$ converge weakly to $\nu^{2}_\star \in \cB$ . In particular, along this subsequence the $\nu_n$ converge weakly to $p^1_\star \nu_\star^1 + p^2_\star \nu_\star^2 \in \cC$.
\end{proof}
By a  simple recursion, Lemma \ref{le:convexcombination} extends immediately to sets $\cC$ of the form 
\begin{align*}
\tilde{\cC} := \big\{  \nu := \sum_{i=1}^N p^i \nu^i  \colon 0 \leq p^{i} \leq 1,\, \sum_{i=1}^N  p^i =1, \,\nu^i  \in \cA_i   \}, 
\end{align*}
for fixed $N$ and sets $\cA_i$ that are all closed under weak convergence of probability measures. Hence we get the  following consequence from Corollary \ref{c:firstproperties} and Lemma \ref{le:GaussianConv}.

\begin{theorem}\label{thm:Gaussian5}
Let $\mu_0$ be a Gaussian measure with mean $m_0 \in \cH$ and covariance operator $C_0 \in \TC$ and let $\mu$ be given by \eqref{e:target}. For any fixed $N$ and for any choice of set $\cA$ as in Theorem \ref{thm:Gaussian} consider the following choice for $\cC$ 
\begin{align*}
\cC := \big\{  \mu := \sum_{i=1}^N p^i \nu^i  \colon 0 \leq p^{i} \leq 1,\, \sum_{i=1}^N  p^i =1, \,\nu^i  \in \cA   \}.
\end{align*}
Then as soon as there exists a single $\nu \in \cA$  with $\Dnm < \infty$ there exists a minimiser of $\nu \mapsto \Dnm$ in $\cC$. 
This minimiser $\nu$ is necessarily equivalent to $\mu_0$ in the sense of measures.
\end{theorem}

\section{Proofs of Main Results}
\label{sec:proofs}

Here we gather the proofs of various results used in the paper
which, whilst the proofs may be of independent interest, their
inclusion in the main text would break from
the flow of ideas related to Kullback-Leibler minimisation

\subsection{Proof of Lemma~\ref{le:CS}}\label{sec:lemmaCS}
The following ``parallelogram identity" (See \cite[Equation (2.2)]{CS}) is easy to check: for any $n,m$
\begin{align}
&\Dkl(\nu_n \| \mu) + \Dkl(\nu_m \|\mu) \notag \\
&=  \,2 \Dkl\bigg(\frac{\nu_n + \nu_m}{2} \bigg\| \mu\bigg) + \Dkl\bigg( \nu_n \bigg\|  \frac{\nu_n + \nu_m}{2} \bigg) +\Dkl\bigg(\nu_m\bigg\|  \frac{\nu_n + \nu_m}{2} \bigg) .\label{e:parallelogra}
\end{align}
By assumption the left hand side of \eqref{e:parallelogra} converges to $ 2 \Dkl (\ns\| \mu)$ as $n,m \to \infty$. Furthermore, the measure $1/2(\nu_n + \nu_m)$ converges weakly to $\ns$ as $n,m \to \infty$ and by lower semicontinuity of $\nu \mapsto \Dnm$  we have 
\begin{equation*}
\liminf_{n,m \to \infty} 2 \Dkl\bigg(\frac{\nu_n + \nu_m}{2} \bigg\| \mu\bigg) \geq 2  \Dkl (\ns\| \mu).
\end{equation*}
By the non-negativity of $\Dkl$ this implies that 
\begin{equation}\label{e:BV1}
\Dkl\bigg( \nu_m \bigg\|  \frac{\nu_n + \nu_m}{2} \bigg)  \to 0 \quad \text{and }\quad\Dkl\bigg(\nu_n\bigg\|  \frac{\nu_n + \nu_m}{2} \bigg) \to 0.
\end{equation}
As we can write 
\begin{equation*}
\|\nu_n - \nu_m \rTV \leq \Big\| \nu_n - \frac{\nu_n + \nu_m}{2} \Big\|_{{\rm {tv}}} + \Big\| \nu_m - \frac{\nu_n + \nu_m}{2} \Big\|_{{\rm {tv}}},
\end{equation*}
equations \eqref{e:BV1} and the Pinsker inequality
\begin{equation*}
\| \nu -\mu \rTV \leq \sqrt{\frac12 \Dnm}
\end{equation*}
(a proof of which can be found in   \cite{cover2012elements}) 
 imply that the sequence is Cauchy with respect to the total variation norm. By assumption the $\nu_n$  converge \emph{weakly} to $\ns$ and this implies convergence in total variation norm.

\subsection{Proof of Lemma \ref{le:Gamma}}

Recall $(\ea, \, \lambda_\alpha, \alpha \geq 1)$ the eigenfunction/eigenvalue 
pairs of $C_0$, as introduced above \eqref{e:projectgamma}.  For any $\alpha, \beta$ we write 
\begin{equation*}
\Gamma_{\alpha,\beta} = \la \ea, \Gamma e_\beta \ra.
\end{equation*}
Then \eqref{e:FHH} states that 
\begin{equation*}
\sum_{1 \leq \alpha, \beta< \infty} \lambda_\alpha \, \lambda_\beta \Gamma_{\alpha,\beta}^2  < \infty.
\end{equation*}
Define $\N_0=\N^2 \setminus\{1, \ldots,N_0  \}^2$.
Then the preceding display implies that, for any $\delta >0$  there exists an $N_0 \geq 0 $ such that
\begin{equation}\label{e:Gamma2}
\sum_{(\alpha, \beta) \in \N_0} \lambda_\alpha \, \lambda_\beta \Gamma_{\alpha,\beta}^2  < \delta^2.
\end{equation}
This implies that for $x = \sum_\alpha x_\alpha \ea \in \cH^1$ we get
\begin{align}\label{e:Gamma1}
\la x, \Gamma x \ra &= \sum_{1 \leq \alpha, \beta < \infty} \Gamma_{\alpha,\beta} x_\alpha x_\beta \notag\\
&= \sum_{1 \leq \alpha, \beta \leq N_0} \Gamma_{\alpha, \beta} x_\alpha x_\beta + 
\sum_{ (\alpha, \beta) \in \N_0} \Gamma_{\alpha, \beta} x_\alpha x_\beta. 
\end{align}
The first term on the right hand side of \eqref{e:Gamma1} can be bounded by 

\begin{align}\label{e:Gamma4}
\bigg| \sum_{1 \leq \alpha, \beta \leq N_0} \Gamma_{\alpha, \beta} x_\alpha x_\beta \bigg|  \leq \max_{1 \leq \alpha, \beta \leq N_0} \big| \Gamma_{\alpha,\beta}\big|  \| x \|_{\cH}^2.
\end{align}
For the second term we get using Cauchy-Schwarz inequality and \eqref{e:Gamma2}
\begin{align}
\bigg| \sum_{ (\alpha, \beta) \in \N_0} \Gamma_{\alpha, \beta} x_\alpha x_\beta \bigg|&= \bigg|  \sum_{ (\alpha, \beta) \in \N_0}   \sqrt{\lambda_\alpha \lambda_\beta} \Gamma_{\alpha, \beta}  \,\frac{  \, x_\alpha x_\beta }{\sqrt{\lambda_\alpha \lambda_\beta}}\bigg|  \notag\\
&\leq \delta \la x, \Cni x\ra.\label{e:Gamma3}
\end{align}
We can conclude from \eqref{e:Gamma1}, \eqref{e:Gamma4}, and \eqref{e:Gamma3} that $\Gamma$ is infinitesimally form-bounded with respect to $\Cni$ (see e.g. \cite[Chapter X.2]{RS2}). In particular, by the KLMN theorem (see \cite[Theorem X.17]{RS2}) the form $Q_\Gamma$ is bounded from below, closed, and there exists a unique self-adjoint operator denoted by $\Cni + \Gamma$ with form domain $\cH^1$ that generates $Q_\Gamma$. 

If $Q_\Gamma$ is strictly positive, then so is $\Cni + \Gamma$  and its inverse  $(\Cni + \Gamma)^{-1}$. As $\Cni +\Gamma$ has form domain $\cH^{1}$ the operator $(\Cni + \Gamma)^{-\frac12} C_0^{-\frac12}$ is bounded on $\cH$ by the closed graph theorem and  it follows that,
as the composition of a trace class operator with two bounded
operators, 
\begin{equation*}
(\Cni + \Gamma)^{-1} =  \big( (\Cni + \Gamma)^{-\frac12} C_0^{-\frac12}  \big)  C_0  \big( (\Cni + \Gamma)^{-\frac12} C_0^{-\frac12} \big)^{\star}    
\end{equation*}
is a trace-class operator. It is hence
the covariance operator of a centred Gaussian measure on $\cH$. It  satisfies the conditions in of the Feldman-Hajek Theorem by assumption.

If  $Q_\Gamma$ is not strictly positive, then the intersection of the spectrum of $\Cni + \Gamma$ with $(-\infty,0]$ is not empty and hence  it cannot be the inverse covariance of a Gaussian measure. 

\subsection{Proof of Theorem~\ref{thm:uniqueness} }\label{sec:proofUniqueness}

We start  the proof  of Theorem \ref{thm:uniqueness} with the following Lemma:
\begin{lemma}\label{le:PhiConvergence}
Let $\nu= \cN(m,C)$ be equivalent to $\mu_0$. For any $\gamma \geq 1$ let $\pi_\gamma\colon \cH \to \cH$ be the orthogonal projector on the space $\cH_\gamma$ introduced in \eqref{e:projectgamma}. Furthermore, assume that $\Phi \colon X \to \R_+$ satisfies the second inequality in \eqref{e:PhiBo}. Then we have 
\begin{equation}\label{e:aaa1}
\lim_{\gamma  \to \infty }\E^{\nu}\big[ \Phi(\pi_\gamma x)\big] = \E^{\nu}\big[\Phi(x) \big].
\end{equation}
\end{lemma}
\begin{proof}
It is a well known property of the white noise/Karhunen-Loeve expansion (see e.g. \cite[Theorem 2.12]{ZDP} ) that $\| \pi_\gamma x -x \|_X \to 0$ $\mu_0$-almost surely, and as $\nu$ is equivalent to $\mu_0$, also $\nu$-almost surely. Hence,
by continuity of $\Phi$ on $X$, $\Phi(\pi_\gamma x)$ converges
$\nu-$almost surely to $\Phi(x)$.


As $\nu(X) =1$ there exists a constant $0 <K_\infty < \infty$ such that  $\nu(\| x\|_{X} \geq K_\infty) \leq \frac{1}{8}$. On the other hand, by the $\nu$-almost sure convergence of $\| \pi_\gamma x - x \|_X$ to $0$
 there exists a $\gamma_\infty \geq 1$ such that for all $\gamma > \gamma_\infty$ we have $\nu\big(\|\pi_\gamma x -x \|_X \geq 1 \big) \leq \frac18$ which implies that
\begin{equation*}
\nu\big( \| \pi_\gamma x  \| \geq K_\infty +1 \big) \leq \frac14 \qquad \text{for all } \gamma \geq \gamma_\infty.
\end{equation*}
For any $\gamma \leq \gamma_\infty$ there exists another  $0 < K_\gamma< \infty$ such that $\nu\big( \| \pi_\gamma x  \| \geq K_\gamma  \big) \leq \frac14$ and hence if we set $K = \max \{K_1, \ldots, K_{\gamma_\infty}, K_{\infty}+1  \}$ we get
\begin{equation*}
\nu\big( \| \pi_\gamma x  \| \geq K \big) \leq \frac14 \qquad \text{for all } \gamma \geq 1.
\end{equation*}
By Fernique's Theorem (see e.g. \cite[Theorem 2.6]{ZDP}) this implies the existence of a $\lambda >0$ such that 
\begin{equation*}
\sup_{\gamma \geq 1}\E^{\nu} \big[ \exp \big(\lambda \| \pi_\gamma x\|^2  \big)  \big]<\infty.
\end{equation*}
Then the desired statement \eqref{e:aaa1} follows from the dominated convergence theorem observing  that \eqref{e:PhiBo} implies the pointwise bound
\begin{equation}
\Phi(x) \leq c_2 \exp( c_3 \| x\|^\alpha_X \big) \leq c_4 \exp( \lambda \| x\|^2_X \big),
\end{equation}
for  $0 < c_4< \infty$ sufficiently large. 
\end{proof}

Let us also recall the following property.
\def\AGS{\cite[Lemma 9.4.5]{AGS}}
\begin{proposition}[\AGS]\label{le:AGS1}
Let $\mu,\nu \in \Me$ be a pair of arbitrary probability measures
on $\cH$ and let $\pi \colon \cH \to \cH$ be a measurable mapping. Then we have
\begin{equation}
\Dkl( \nu \circ \pi^{-1}\| \mu \circ \pi^{-1}) \leq  \Dnm.
\end{equation}

\end{proposition}

\begin{proof}[Proof of Theorem \ref{thm:uniqueness}]
As above in \eqref{e:projectgamma}, let $(\ea, \alpha \geq 1)$ be the basis $\cH$ consisting of eigenvalues of $C_0$ with the corresponding eigenvalues $(\lambda_\alpha, \alpha \geq 1)$.  For $\gamma \geq 1$ let $\pg\colon \cH \to \cH$ be the orthogonal projection on $\Hg := \spa(e_1, \ldots , e_\gamma)$. Furthermore, for $\alpha \geq 1$ and $x \in \cH$ let $\xa(x) = \la x, \ea \ra_{\cH}$.  Then we can identify $\Hg$ with $\R^\gamma$ through the bijection
\begin{equation}\label{e:uniqueness1}
\R^\gamma \ni \Xi_\gamma = (\xi_1, \ldots, \xi_\gamma) \mapsto \sum_{\alpha=1}^\gamma \xa \ea.
\end{equation}
The identification \eqref{e:uniqueness1} in particular gives a natural way to define the $\gamma$-dimensional Lebesgue measure $\cL^\gamma$ on $\Hg$.

Denote by $\mu_{0;\gamma} = \mu_0 \circ \pg^{-1}$ the projection of $\mu_0$ on $\Hg$.  We also define $\mg$ by 
\begin{equation*}
\frac{d\mg}{d \mu_{0;\gamma}}(x) = \frac{1}{Z_\gamma} \exp \big(-\Phi(x)  \big) ,
\end{equation*}
where $Z_\gamma = \E^{\mu_{0,\gamma}}\big[ \exp \big(-\Phi(x)  \big) \big]$.
Note that in general $\mg$ does not coincide with the measure $\mu \circ \pg$. The Radon-Nikodym density of $\mg$ with respect to $\cL^\gamma$ is given by 
\begin{equation*}
\frac{d\mu_{\gamma}}{d \cL^\gamma}(x) = \frac{1}{\tilde{Z}_\gamma} \exp \big(-\Psi(x)  \big),
\end{equation*}
where $\Psi(x) = \Phi(x) + \frac{1}{2} \la  x,x \ra_{\cH^1}$ and   the normalisation constant is given by 
\begin{equation*}
\tilde{Z}_\gamma = Z_\gamma (2 \pi)^{\frac{\gamma}{2}} \prod_{\alpha=1}^\gamma \sqrt{\lambda_\alpha}.
\end{equation*}
According to the assumption the function $\Psi(x) - \frac{\kappa}{2} \la x,x \ra_{\cH^1}$ is convex on $\cH_\gamma$ which implies that for any $x_1,x_2 \in \Hg$ and for $t \in [0,1]$ we have
\begin{equation*}
\Psi\big( (1-t) x_1 + t x_2 \big) \leq (1-t) \Psi(x_1) + t\Psi(x_2) -  \kappa \frac{t(1-t)}{2} \|x_1 - x_2 \|_{\cH^1}^2. 
\end{equation*}

Let us also define the projected measures $\nu_{i;\gamma} := \nu_i \circ \pg^{-1}$ for $i = 1,2$. By  assumption the measures $\nu_{i}$ equivalent to $\mu_0$ and therefore the projections $\nu_{i;\gamma}$ are equivalent to $\mu_{0;\gamma}$.  In particular, the $\nu_{i;\gamma}$ are non-degenerate Gaussian measures on $\Hg$. Their covariance operators are given by $C_{i;\gamma} := \pg C_{i} \pg$ and the means by $m_{i;\gamma} = \pg m_i$. 

There is a convenient coupling between the $\nu_{i;\gamma}$. Indeed, set
\begin{equation}
\Lambda_\gamma = C_{2;\gamma}^{\frac12} \big( C_{2;\gamma}^{\frac12} C_{1;\gamma}  C_{2,\gamma}^{\frac12} \big)^{-\frac12}  C_{2;\gamma}^{\frac12} \in \cL(\cH_\gamma,\cH_\gamma).
\end{equation}
The operator $\Lambda_\gamma$ is symmetric and strictly positive on $\cH_\gamma$. Then define for $x \in \cH_\gamma$
\begin{equation}
\tLg (x) := \Lambda_\gamma (x - m_{1;\gamma})  + m_{2;\gamma}  .
\end{equation}
Clearly, if $x \sim \nu_{1,\gamma}$ then $\tLg(x) \sim \nu_{2,\gamma}$. Now for any $ t\in (0,1)$ we define the interpolation $\tLgt (x) = (1-t) x + t \tLg(x)$ and the approximate interpolating measures $\ntg $ for $t \in (0,1)$ as push-forward measures
\begin{equation}
\ntg := \nu_{1,\gamma} \circ \tLgt^{-1}.
\end{equation}
From the construction it follows that the $\ntg= \cN(m_{t,\gamma},C_{t,\gamma}) $ are non-degenerate Gaussian measures on $\cH_\gamma$. Furthermore, if the means $m_1$ and $m_2$ coincide, then we have $m_{1,\gamma} = m_{2,\gamma} = m_{t,\gamma}$ for all $t \in (0,1)$ and in the same way, if the covariance operators $C_1$ and $C_2$ coincide, then we have $C_{1,\gamma}=C_{2,\gamma} = C_{t,\gamma}$  for all $t \in (0,1)$.

\vskip2ex
As a next step we will establish that for any $\gamma$ the function 
\begin{equation*}
t \mapsto \Dkl(\ntg \| \mg)
\end{equation*}
is  convex. 
To this end it is useful to write 
\begin{equation}\label{e:unique0}
\Dkl(\ntg \| \mg) = \mathscr{H}_\gamma(\ntg) +\mathscr{F}_\gamma(\ntg) + \log(\tilde{Z}_\gamma)
\end{equation}
where $\mathscr{F}_\gamma(\ntg) = \E^{\ntg} \big[ \Psi(x)  \big]$ and 
\begin{equation*}
\mathscr{H}_\gamma(\ntg) = \int_{\Hg}  \frac{d \ntg}{d \cL^\gamma}(x)  \log \bigg( \frac{d \ntg}{d \cL^\gamma}(x)  \bigg)   d\cL^\gamma(x). 
\end{equation*}
Note that $\mathscr{H}_\gamma(\ntg)$ is completely independent of the measure $\mu_0$. Also note that $\mathscr{H}_\gamma(\ntg)$, the entropy of $\ntg$, can be negative because the Lebesgue measure is not a probability measure.

We will treat the terms $ \mathscr{H}_\gamma(\ntg) $ and $F_\gamma(\ntg) $ separately. The treatment of $F_\gamma$ is straightforward using the $(-\kappa)$-convexity of $\Psi$ and the coupling described above. Indeed, we can write
\begin{align}
&\mathscr{F}_\gamma(\ntg) =  \E^{\ntg} \big[ \Psi(x)  \big]\notag\\
&= \E^{\nu_{1,\gamma}}\big[ \Psi \big((1-t) x + t \tLg( x) \big) \big] \notag\\
&\leq (1-t)  \E^{\nu_{1,\gamma}}\big[ \Psi(x) \big] +  t  \E^{\nu_{1,\gamma}}\big[ \Psi(\tLg( x) ) \big]  - \kappa \frac{t (1-t)}{2} \E^{\nu_{1,\gamma}}  \| x - \tLg( x) \|_{\cH^1}^2\notag\\
& \leq (1-t) \mathscr{F}_\gamma \big( \nu_{1,\gamma}\big)  +  t   \mathscr{F}_\gamma \big( \nu_{2,\gamma}\big) -\kappa\frac{ t (1-t)}{2} \E^{\nu_{1,\gamma}}  \| x - \tLg( x) \|_{\cH^1}^2.\label{e:unique1}
\end{align}
Note that this argument does not make use of any specific properties of the mapping $x \mapsto \tLg (x)$, except that it maps $\mu_{1;\gamma}$ to  $\mu_{2;\gamma}$. The same argument would work for different mappings with this property.

To show the convexity of the functional $\mathscr{H}_\gamma$ we will make use of the fact that the matrix $\Lambda_\gamma$ is symmetric and strictly positive. For convenience, we introduce the notation 
\begin{equation*}
 \rho(x) = \frac{\nu_{1;\gamma}}{d\cL^\gamma}(x) \qquad \rho_t(x) := \frac{d\ntg}{d\cL^\gamma}(x).
\end{equation*}

Furthermore, for the moment we write $F(\rho) = \rho \log (\rho)$.  By the change of variable formula we have
\begin{equation*}
\rho_t(\tLg(x)) = \frac{\rho(x)}{\det\big((1-t) \Id_\gamma  + t\Lambda_\gamma \big) }, 
\end{equation*}
where we denote by $\Id_\gamma$ the identity matrix on $\R^\gamma$. Hence we can write 
\begin{align}
\mathscr{H}_\gamma(\ntg) &= \int_{\Hg}  F\big(\rho_t(x)\big) d\cL^\gamma(x) \notag\\
& = \int_{\Hg}    F \bigg( \frac{\rho(x)}{\det\big((1-t) \Id_\gamma + t\Lambda_\gamma \big) }\bigg) \det\big((1-t) \Id_\gamma + t\Lambda_\gamma \big)  d\cL^\gamma(x)\notag
\end{align}
For a diagonalisable matrix $\Lambda$ with non-negative eigenvalues the mapping $[0,1] \ni t \mapsto \det((1-t) \Id + t \Lambda_\gamma)^{\frac{1}{\gamma}}$ is concave, and as the map $s \mapsto F(\rho/s^d)s^d$ is non-increasing the resulting map is convex in $t$. Hence we get
\begin{align}
\mathscr{H}_\gamma(\ntg)& \leq (1-t)  \int_{\Hg}    F \big( \rho(x)\big)    d\cL^\gamma(x)+  t  \int_{\Hg}    F \bigg(  \frac{\rho(x)}{ \Lambda_\gamma}  \bigg)\, \det\big( \Lambda_\gamma \big)  d\cL^\gamma(x)\notag\\
&= (1-t) \mathscr{H}_\gamma\big(\nu_{1;\gamma}\big)  + tH_{\gamma}\big(\nu_{2;\gamma}\big) \label{e:unique2}
\end{align}
  Therefore, combining \eqref{e:unique0}, \eqref{e:unique1} and \eqref{e:unique2} we obtain for any $\gamma$ that 
  \begin{align}\label{e:unique3}
  \Dkl\big(\ntg \big\| \mu_{\gamma} \big) \leq&  (1-t) \Dkl\big(\nu_{1,\gamma} \big\| \mu_{\gamma}  \big) +  t \Dkl\big(\nu_{2,\gamma} \big\| \mu_{\gamma}  \big) \notag\\
  &- \kappa \frac{t (1-t)}{2}  \E^{\nu_{1,\gamma}}  \| x - \tLg( x) \|_{\cH^1}^2.
  \end{align}

\vskip2ex
 It remains to pass to the limit $\gamma \to \infty$ in \eqref{e:unique3}. First  we establish that for $i=1,2$ we have $\Dkl\big(\nu_{i,\gamma} \big\| \mu_{\gamma}  \big)  \to \Dkl\big(\nu_{i} \big\| \mu  \big)$. In order to see that we write
 \begin{equation}\label{e:Unique4}
 \Dkl\big(\nu_{i,\gamma} \big\| \mu_{\gamma}  \big) = \Dkl\big(\nu_{i,\gamma} \big\| \mu_{0, \gamma}  \big) + \E^{\nu_{i,\gamma}}\big[ \Phi(x) \big] + \log(Z_\gamma),
 \end{equation}
 and a similar identity holds for $ \Dkl\big(\nu_{i} \big\| \mu  \big)$. The Gaussian measures $\nu_{i,\gamma}$ and $\mu_{0,\gamma} $ are projections of the measures $\nu_{i}$ and $\mu_0$ and hence they converge weakly as probability measures on $\cH$ to these measures as $\gamma \to \infty$. Hence the lower-semicontinuity of the Kullback-Leibler divergence (Proposition \ref{le:firstproperties}) implies that for $i=1,2$
  \begin{equation*}
\liminf_{\gamma \to \infty } \Dkl\big(\nu_{i,\gamma} \big\| \mu_{\gamma}  \big)  \geq \Dkl\big(\nu_{i} \big\| \mu_0  \big).
\end{equation*}
 On the other hand the Kullback-Leibler divergence is monotone under projections (Proposition \ref{le:AGS1}) and hence we get
\begin{equation*}
\limsup_{\gamma \to \infty } \Dkl\big(\nu_{i,\gamma} \big\| \mu_{\gamma}  \big)  \leq \Dkl\big(\nu_{i} \big\| \mu_0  \big),
\end{equation*} 
which established the convergence of the first term in \eqref{e:Unique4}. The convergence of the $Z_\gamma = \E^{\mu_{0,\gamma}}\big[ \exp \big(-\Phi(x)  \big) \big]$ and of the  $\E^{\nu_{i,\gamma}}\big[ \Phi(x) \big]$ follow from Lemma \ref{le:PhiConvergence} and the integrability assumption \eqref{e:PhiBo}.
%
%
 
 In order to pass to the limit $\gamma \to \infty$ on the left hand side of \eqref{e:unique3} we note that for fixed $t \in (0,1)$ the measures $\ntg$ form a tight family of measures on $\cH$. Indeed, by weak convergence the families of measures $\nu_{1,\gamma}$ and $\nu_{2,\gamma}$ are tight on $\cH$. Hence, for every $\eps>0$ there exist compact in $\cH$ sets $K_1$ and $K_2$ such that for $i = 1,2$ and for any $\gamma$ we have $\nu_{i,\gamma}(K_i^c) \leq \eps$.  For a fixed $t \in (0,1)$ the set 
 \begin{equation*}
 K_t := \{  x = (1-t) x_1 + t x_2 \colon \quad x_1 \in K_1 , \, x_2 \in K_2 \}
 \end{equation*}
 is compact in $\cH$ and we have, using the definition of $\ntg$ that
 \begin{equation*}
 \ntg(K_t^c) \leq \nu_{1,\gamma}(K_1^c) + \nu_{2,\gamma}(K_2^c) \leq 2 \eps, 
 \end{equation*}
 which shows the tightness. Hence we can extract a subsequence that converges to a limit $\nu_{t}^{1 \to 2}$. This measure is Gaussian by Lemma \ref{le:GaussianConv} and by construction its mean coincides with $m_1$ if $m_1 = m_2$ and in the same way its covariance coincides with $C_1$ if $C_1 = C_2$. By lower semicontinuity of the Kullback-Leibler divergence (Proposition \ref{le:firstproperties}) we get
 \begin{equation}
  \Dkl\big(\nu_{t}^{1 \to 2} \big\| \mu \big) \leq \liminf_{\gamma \to \infty}  \Dkl\big(\ntg \big\| \mu_{\gamma} \big)  
 \end{equation}
 Finally, we have 
\begin{equation}\label{e:unique7}
\limsup_{\gamma \to \infty } \E^{\nu_{1,\gamma}}  \| x - \tLg( x) \|_{\cH^1}^2 := K > 0.
\end{equation}
In order to see this note that the measures $\rho_\gamma := \nu_{1,\gamma} [\Id + \tLg]^{-1}$ form a tight family of measures on $\cH \times \cH$. Denote by $\rho$ a limiting measure. This measure is a coupling of $\nu_1$ and $\nu_2$ and hence if these measures do not coincide we have
\begin{equation*}
\E^{\rho} \|x-y \|_{\cH^1}^2   >0.
\end{equation*}
Hence, the desired estimate \eqref{e:unique7} follows from Fatou's Lemma. This finishes the proof.
\end{proof}

\appendix
\section{Appendix}
.

\subsection{Proof of Proposition \ref{le:firstproperties}}

For completeness we give a proof of the well-known Proposition 
\ref{le:firstproperties}, following closely 
the exposition in \cite[Lemma 1.4.2]{DE};
see also \cite[Lemma 9.4.3]{AGS}.

We start by recalling the Donsker-Varadhan variational formula
\begin{equation}
\Dnm = \sup_{\Theta}  \, \En \Theta - \log \E^{\mu} e^\Theta ,
\label{e:DklVariational}
\end{equation} 
where the supremum can be taken either over all bounded continuous functions or all bounded measurable functions $\Theta \colon \cH \to \R$. Note that as soon as $\nu$ and $\mu$ are equivalent, the supremum is realised  for 
$\Theta = \log \big( \frac{d \nu}{ d \mu } \big)$.

We first prove lower semi-continuity. 
For any bounded and  \emph{continuous}   $\Theta\colon \cH \to \R$ the mapping $(\nu, \mu) \mapsto  \En \Theta - \log \E^{\mu} e^\Theta$ is continuous with respect to weak convergence of $\nu$ and $\mu$. Hence, by \eqref{e:DklVariational} the mapping $(\nu, \mu) \mapsto \Dnm$ is lower-semicontinuous as 
the pointwise supremum of continuous mappings.

We now prove compactness of sub-levelsets.
By the lower semi-continuity of $\nu \mapsto \Dnm$  and  Prokohorov's Theorem 
\cite{billingsley2009convergence} 
it is sufficient to show that for any $M < \infty$  the set $\cB := \{ \nu \colon \Dnm \leq M \}$ is tight. 
The measure $ \mu$ is inner regular, and therefore for any $0<\delta \leq 1$ there exists a compact set $K_\delta$ such that $\mu(K_\delta^c) \leq \delta$. Then choosing 
$\Theta = \mathbf{1}_{K_\delta^c} \log\big(1 + \delta^{-1} \big)  $ 
in \eqref{e:DklVariational} we get, for any $\nu \in \cB$,
\begin{align*}
\log\big(1 + \delta^{-1} \big)   \nu(K^c_\delta)& =  \En \Theta\\
 &\leq M  +\log(  \E^{\mu} e^{\Theta}  \big)\\
& = M  + \log \Big( \mu (K_\delta) + \mu(K_\delta^c) \big( 1 + \delta^{-1} \big)  \Big)\\
& \leq  M  + \log \Big(1  + \big( \delta + 1 \big)  \big).
\end{align*}
Hence, if  for  $\eps>0$ we choose $\delta$ small enough to ensure that 
\begin{equation*}
\frac{M  + \log (3)  }{\log\big(1 + \delta^{-1} \big) } \leq \eps,
\end{equation*}
we have, for all $\nu \in \cB$, that  $\nu(K_\delta^c) \leq \eps$.

\subsection{Some properties of Gaussian measures}\label{A:Gaussian}
The following Lemma
summarises some useful facts about weak convergence of Gaussian measures.
\begin{lemma}\label{le:GaussianConv}
Let $\nu_n$ be a sequence of Gaussian measures on $\cH$ with mean $m_n \in \cH$ and covariance operators $C_n$. 
\begin{enumerate}
\item If the $\nu_n$ converge weakly to $\ns$, then $\ns$ is also Gaussian. 
\item If $\ns$ is Gaussian with mean $m_\star$ and covariance operator $\Cs$, then  $\nu_n$ converges weakly to $\ns$ if and only if the following conditions are satisfied:
\begin{itemize}
\item[a)] $\|m_n -m_\star \|_{\cH}$ converges to $0$. 
\item[b)] $\| \sqrt{C_n} - \sqrt{C_\star} \|_{\HS} $ converges to $0$.  
\end{itemize}
\item Condition b) can be replaced by the following condition: 
\begin{itemize}
\item[b')] $\|C_n -C_\star\|_{\cL(\cH)}$ and $ \E^{\nu_n} \|x\|_{\cH}^2 - \E^{\ns} \|x\|_{\cH}^2$ converge to $0$   . 
\end{itemize}
\end{enumerate}
\end{lemma}
\begin{proof}
1.) Assume that $\nu_n$ converges weakly to $\nu$. Then for any continuous linear functional $\ph \colon \cH \to \R$ the push-forward measures $ \nu_n \circ \ph^{-1}$ converge weakly to $\nu \circ \ph^{-1}$. The measures $ \nu_n \circ \ph^{-1}$ are Gaussian measures on $\R$. For one-dimensional Gaussians a simple calculation with the Fourier transform (see e.g. \cite[Prop. 1.1]{LeGall}) shows that weak limits are necessarily Gaussian and weak convergence is equivalent to convergence of mean and variance. Hence $\nu \circ \ph^{-1}$ is Gaussian, which in turn implies that $\nu$ is Gaussian. 
Points 2.) and 3.) are established in \cite[Chapter 3.8]{Bog}. 
\end{proof}

As a next step we recall the Feldman-Hajek Theorem as
proved in \cite[Theorem 2.23]{ZDP}.

\begin{proposition}\label{prop:FH}
Let $\mu_1 = \cN(m_1, C_1)$ and $\mu_2= \cN(m_2,C_2)$ be two Gaussian measures on $\cH$. The measures $\mu_1$ are either singular or equivalent. They are equivalent if and only if the following three assumptions hold:
\begin{enumerate}
\item The Cameron Martin spaces $C_1^{\frac12} \cH$ and $C_2^{\frac12} \cH$ 
are norm equivalent spaces with, in general, different scalar products generating 
the norms  -- we denote the space by $\cH^1$.
\item The means satisfy $m_1 - m_2 \in \cH^1$.
\item The operator $\big(C_1^{\frac{1}{2}} C_2^{- \frac{1}{2}}\big) \big(C_1^{\frac{1}{2}} C_2^{- \frac{1}{2}}\big)^{\star} - \Id$ is a Hilbert-Schmidt operator on $\cH$.
\end{enumerate}
\end{proposition}
\begin{remark}
Actually, in \cite{ZDP} item 3) is stated as  $\big(C_2^{-\frac{1}{2}} C_1^{ \frac{1}{2}}\big) \big(C_2^{-\frac{1}{2}} C_1^{ \frac{1}{2}}\big)^{\star} - \Id$ is a Hilbert-Schmidt operator on $\cH$. We find the formulation in item 3) more
useful and the fact that it is well-defined follows since $C_1^{\frac{1}{2}} C_2^{- \frac{1}{2}}$ is
the adjoint of $C_2^{-\frac{1}{2}} C_1^{ \frac{1}{2}}$. The two conditions
are shown to be equivalent in \cite[Lemma 6.3.1 (ii)]{Bog}. 
\end{remark}
The methods used within
the proof of the Feldman-Hajek Theorem, as given in \cite[Theorem 2.23]{ZDP},
are used below to prove the following characterisation of convergence with respect to total variation norm for Gaussian measures.
\begin{lemma}\label{le:FH}
For any $n \geq 1$ let $\nu_n$ be a  Gaussian measure on $\cH$ with covariance operator $C_n$ and mean $m_n$ and let $\ns$ be a Gaussian measure with covariance operator $\Cs$ and mean $m_\star$. Assume that the measures $\nu_n$  converge to  $\ns$ in total variation. Then we have
\begin{equation}\label{e:FH}
\big\| \Cs^{\frac12} \big( C_n^{-1} - \Cs^{-1}\big) \Cs^{\frac12} \big\|_{\HS}   \to 0 \quad \text{and} 
\quad  \| m_n - m_\star \|_{\cH^1} \to 0.
\end{equation} 
\end{lemma}

In order to proof Lemma~\ref{le:FH} we recall that for two probability measures $\nu$ and $\mu$ the \emph{Hellinger distance} is defined as 
\begin{equation*}
\DHel(\nu;\mu)^2  = \frac12 \int \left( \sqrt{ \frac{d \nu}{ d\lambda}(x)} - \sqrt{ \frac{d \mu}{ d\lambda}(x)}\right)^2 d \lambda(dx),
\end{equation*}
where $\lambda$ is a probability measure on $\cH$ such that $\nu \ll \lambda$ and $\mu \ll \lambda$. Such a $\lambda$ always exists (average $\nu$ and $\mu$
for example) and the value does not depend on the choice of $\lambda$.

For this we need the Hellinger integral 
\begin{equation}\label{e:Hellinger}
\Hnm =   \int \sqrt{\frac{d \mu}{ d \lambda} (x)} \sqrt{ \frac{d \nu}{d\lambda}(x) } \, \lambda(dx) = 1-\DHel(\nu;\mu)^2.
\end{equation}
We recall some properties of $\Hnm$:
\def\ZDP{\cite[Proposition 2.19]{ZDP}}
\begin{lemma}[\ZDP]\label{le:Hellinger}
\begin{enumerate}
\item For any two probability measures $\nu$ and $\mu$ on $\cH$ we have $0 \leq \Hnm \leq 1$. We have $\Hnm = 0$ if and only if $\mu$ and $\nu$ are singular, and $\Hnm=1$ if and only if $\mu = \nu$. 
\item Let $\tilde{\cF}$ be as sub-$\sigma$-algebra of $\cF$ and denote by $H_{\tilde{\cF}}(\nu,\mu)$ the Hellinger integrals of the restrictions of $\nu$ and $\mu$ to $\tilde{\cF}$. Then we have 
\begin{equation}\label{e:HellingerMono}
H_{\tilde{\cF}}(\nu,\mu) \geq \Hnm.
\end{equation}
\end{enumerate}
\end{lemma}

\begin{proof}[Proof of Lemma \ref{le:FH}]
Before commencing the proof we demonstrate
the equivalence of the Hellinger and total variation
metrics. On the one hand the elementary inequality 
$(\sqrt{a} - \sqrt{b})^2 \leq |a - b|$ which holds for any $a,b \geq 0$ immediately yields that
\begin{equation*}
\DHel(\nu;\mu)^2 \leq \frac12 \int \left| \frac{d \nu}{d \lambda}(x) - \frac{d \mu}{d \lambda}(x)\right|  \lambda(dx) = \Dtv(\nu,\mu).
\end{equation*}
On the other hand  the elementary equality $(a-b)=(\sqrt{a} - \sqrt{b})(\sqrt{a} + \sqrt{b})$, together with the Cauchy-Schwarz inequality, yields 
\begin{align}
\Dtv(\nu;\mu) &= \frac12 \int  \left| \frac{d \nu}{d \lambda}(x) - \frac{d \mu}{d \lambda}(x)\right|  \lambda(dx) \notag \\ 
&\leq \DHel(\nu;\mu)  \int \left( \sqrt{\frac{d \nu}{d \lambda}(x)}+ \sqrt{\frac{d \mu}{d \lambda}(x)}\right)^2  \lambda(dx)  \leq 4 \DHel(\nu;\mu).\notag
\end{align}
This justifies study of the Hellinger integral to
prove total variation convergence.

We now proceed with the proof.
We first treat the case of centred measures, i.e. we assume that $m_n = m_\star =0$. For $n$ large enough $\nu_n$ and $\ns$ are equivalent and therefore their Cameron-Martin spaces  coincide as sets and in particular the operators $\Cs^{-\frac12}C_n^{\frac{1}{2}}$ are defined on all of $\cH$ and invertible. By 
Proposition \ref{prop:FH} they are invertible bounded operators on $\cH$. Denote by $R_n$ the operator $(\Cs^{-\frac12} C_n^{\frac{1}{2}}) (\Cs^{-\frac12} C_n^{\frac{1}{2}})^\star$. This shows in particular, that the  expression \eqref{e:FH} makes sense, as it can be rewritten as 
 \begin{equation*}
 \big\|  R_n^{-1} -\Id \big\|_{\HS}^2   \to 0.
 \end{equation*}

Denote by $ (\ea, \alpha \geq 1) $ \footnote{Use of the same notation as for the
eigenfunctions and eigenvectors of $C_0$ elsewhere should not cause confusion}  
the orthonormal basis of $\cH$ consisting of eigenvectors of the operator $\Cs$ and by $(\lambda_\alpha, \alpha \geq 1) $ the corresponding sequence of eigenvalues. For any $n$ the operator $R_n$ can be represented in the basis $(\ea)$ by the matrix $(r_{\alpha,\beta;n})_{1 \leq \alpha,\beta < \infty}$ where 
\begin{equation*}
r_{\alpha,\beta;n} = \frac{\la C_n \ea, e_\beta \ra}{\sqrt{\lambda_\alpha \,\lambda_\beta}}.
\end{equation*}
For any $\alpha \geq 1$ define the linear functional 
\begin{equation}\label{e:FH2}
\xa(x) = \frac{\la x, \ea \ra}{\sqrt{\lambda_\alpha}}   \qquad x \in \cH.
\end{equation}
By definition, we have for all $\alpha,\beta$ that 
\begin{align}
\E^{\ns}\big[ \xa(x)\big] &=0, \qquad   &\E^{\nu_n}\big[ \xa(x)\big] &= 0,\notag \\
 \quad  \E^{\ns}\big[ \xa(x) \xi_\beta(x) \big] &= \delta_{\alpha,\beta}, \qquad \text{and} \quad &\E^{\nu_n}\big[ \xa(x) \xi_\beta(x) \big] &= r_{\alpha,\beta;n}.\label{e:FH3}
\end{align}
For any $\gamma \geq 1$ denote by $\cF_\gamma$ the $\sigma$-algebra generated by $(\xi_1, \ldots \xi_\gamma)$. Furthermore, denote by $R_{\gamma;n}$ and $I_\gamma$ the matrices $(r_{\alpha,\beta;n})_{1 \leq \alpha,\beta \leq \gamma}$ and $(\delta_{\alpha,\beta})_{1 \leq \alpha,\beta \leq \gamma}$. With this notation \eqref{e:FH3} implies that we have 
\begin{equation*}
\frac{d\nu_n \big|_{\cF_\gamma}}{ d\ns \big|_{\cF_\gamma}} = \frac{1}{\sqrt{\det(R_{\gamma;n})}} \exp \Big(- \frac{1}{2} \sum_{\alpha,\beta \leq \gamma } \xa \xi_\beta \big( \big(R_{\gamma;n}^{-1}\big)_{\alpha,\beta} -\delta_{\alpha,\beta}\big)   \Big),
\end{equation*}
and in particular we get the Hellinger integrals
\begin{equation*}
H_{\cF_\gamma} \big( \nu_n; \ns \big) =  \frac{ (\det R_{\gamma,n}^{-1} )^{\frac14}} { \Big( \det\Big( \frac{I_\gamma + R_{\gamma;n}^{-1}}{2} \Big)\Big)^{\frac12} }.
\end{equation*}
Denoting by $\big( \lambda_{\alpha; \gamma ;n} , \alpha = 1 , \ldots, \gamma \big) $ the eigenvalues of $R_{\gamma,n}^{-1}$ this expression can be rewritten as 
\begin{equation}\label{e:FH5}
- \log \big(  H_{\cF_\gamma}  \big( \nu_n; \ns \big)\big)  = \frac14 \sum_{\alpha=1}^\gamma  \log \frac{(1 + \lambda_{\alpha;\gamma;n})^2}{4 \lambda_{\alpha; \gamma; n}} \leq - \log \big(  H(\nu_n; \ns)\big),
\end{equation}
where we have used equation \eqref{e:HellingerMono}. The  the right hand side of \eqref{e:FH5} goes to zero as $n \to \infty$ and in particular, it is bounded by $1$ for $n$ large enough, say for $n \geq n_0$. Hence there exist constants $0<K_1,K_2<\infty$ such that for all $n \geq n_0$, and all $\gamma, \alpha$ we have $K_1 \leq \lambda_{\alpha; \gamma;n} \leq K_2$. There exists a third constant $K_3>0$ such that for all $\lambda \in [K_1,K_2]$ we have
\begin{equation*}
(1 - \lambda)^2  \leq \frac{K_3}{4} \log   \frac{(1 + \lambda)^2}{4 \lambda}  .
\end{equation*}
Hence, we can conclude that  for $n \geq n_0$
\begin{align*}
\big\|R_{\gamma,n}^{-1} - I_\gamma  \big\|_{\mathcal{HS}(\R^\gamma)}^2 = \sum_{\alpha=1}^\gamma \big| \lambda_{\alpha;\gamma;n} -1 \big|^2 \leq - K_3  \log \big(  H(\nu_n; \ns)\big).
\end{align*}
As this bound holds uniformly in $\gamma$ the claim is proved in the case $m_n = m_\star =0$.

As a second step let us treat the case where $m_n$ and $m_\star$ are arbitrary but the covariance operators coincide, i.e. for all $n \geq1$ we have $C_n = C_\star=: C$. As above, let $ (\ea, \alpha \geq 1) $ the orthonormal basis of $\cH$ consisting of eigenvectors of the operator $C$ and by $(\lambda_\alpha, \alpha \geq 1) $ the corresponding sequence of eigenvalues. Furthermore, define the random variable $\xa$ as above in \eqref{e:FH2}.  Then we get the identities 
\begin{align}
\E^{\ns}\big[ \xa(x)\big] &=\frac{m_{\star;\alpha}}{\sqrt{\lambda_\alpha}}, \qquad   &\E^{\nu_n}\big[ \xa(x)\big] &= \frac{m_{n;\alpha}}{\sqrt{\lambda_\alpha}},\notag \\
 \quad   \cov^{\ns} \big( \xa(x) ,   \xi_\beta(x) \big) &= \delta_{\alpha,\beta}, \quad \text{and} \quad &\cov^{\nu_n}\big( \xa(x) \xi_\beta(x) \big) &=  \delta_{\alpha,\beta},\notag
\end{align}
where $\cov^{\ns}$ and $\cov^{\nu_n}$ denote the covariances with respect to the measures $\ns$ and $\nu_n$.   Here we have set $m_{\star;\alpha}:=\la m_\star, \ea \ra$ and $m_{n;\alpha}:=\la m_n, \ea \ra$. Denoting as above by $\cF_\gamma$ the $\sigma$-algebra generated by $(\xi_1, \ldots, \xi_\gamma)$ we get for any $\gamma\geq1$
\begin{equation}\label{e:FH11}
H_{\cF_\gamma} \big( \nu_n; \ns \big) =  \exp\Big( - \frac18 \sum_{\alpha=1}^\gamma  \frac{1}{\lambda_\alpha}\big| m_{\star;\alpha} - m_{n;\alpha} \big|^2 \Big).
\end{equation}
Noting that $\| m_n - m_\star \|_{\cH^1}^2 = \sum_{\alpha \geq 1} \frac{1}{\lambda_\alpha}  \big|m_{n;\alpha} - m_{\star;\alpha} \big|^2$ 
and reasoning as above in \eqref{e:FH5} we get that $\| m_n - m_\star \|_{\cH^1}^2 \to 0$.

The general case of arbitrary $m_n$,$m_\star$, $C_n$, and $\Cs$ can be reduced to the two cases above. Indeed, assume that $\nu_n$ converges to $\ns$ in total variation. After a translation which does not change the total variation distance, we can assume that  $m_\star=0$. Furthermore, by symmetry if the the measures $\cN(m_n, C_n)$  converge to $\cN(0,\Cs)$, in total
variation then so do the measures $\cN(-m_n,C_n)$. A coupling 
argument, which we now give,
shows that then the Gaussian measures  $\cN(0,2C_n)$  converge to  $\cN(0,2C_\star)$, also in total variation. Let $(X_1,Y_1)$ be random variables with $X_1 \sim \cN(m_n,C_n)$ and $Y_1 \sim \cN(0, \Cs)$ and $\mathbb{P}(X_1  \neq Y_1) = \| \cN(m_n, C_n)-  \cN(0, \Cs)\rTV $ and in the same way let  let $(X_2,Y_2)$ be independent from $(X_1,Y_1)$ and such that  $X_2 \sim \cN(-m_n,C_n)$ and $Y_2 \sim \cN(0, \Cs)$ with $\mathbb{P}(X_2  \neq Y_2) = \|  \cN(-m_n, C_n)- \cN(0, \Cs))\rTV$. Then we have $X_1 + X_2 \sim \cN(0,2C_n)$, $Y_1 + Y_2 \sim \cN(0, 2 C_\star)$ and 
\begin{align*}
\|\cN(0, 2C_n) -  \cN(0, 2\Cs)\rTV &=
\mathbb{P} (X_1 + X_2 \neq Y_1 + Y_2)\\
&\leq \mathbb{P} (X_1  \neq Y_1 ) + \mathbb{P} ( X_2 \neq  Y_2) \\
&=  2  \| \cN(m_n, C_n) -  \cN(0, \Cs)\rTV.
\end{align*}
Hence we can apply the first part of the proof
to conclude that the desired conclusion concerning the covariances holds.

We now turn to the means.
From the fact that $\cN(m_n,C_n)$ and $\cN(0,C_n)$ converge to $\cN(0,\Cs)$ in total variation we can conclude by the triangle inequality that $\| \cN(m_n, C_n)  - \cN(0, C_n)\rTV \to 0$ and hence $\log H(\cN(m_n, C_n)  , \cN(0, C_n)) \to 0$.
By \eqref{e:FH11} this implies  that
\begin{align*}
\| C_n^{-\frac12} m_n \|_{\cH}  \leq 8  \log H(\cN(m_n, C_n)  , \cN(0, C_n))  \to 0.
\end{align*}
Furthermore, the convergence of $\big\| C_\star^{\frac12}(C_n^{-1} - C_{\star}^{-1})  C_\star^{\frac12} \big\|_{\HS} =\big\| (C_\star^{\frac12}C_n^{-\frac12} ) (C_\star^{\frac12}C_n^{-\frac12} \big)^{\star} - \Id \big\|_{\HS} $ implies that $\sup_{n \geq 1} \| C_\star^{-\frac12} C_n^{\frac12} \|_{\cL(\cH)} < \infty$. So we can conclude that as desired
\begin{equation*}
\|m_n \|_{\cH^1} \leq \Big( \sup_{n \geq 1} \| C_\star^{-\frac12} C_n^{\frac12} \|_{\cL(\cH)} \Big) \, \| C_n^{-\frac12} m_n \|_{\cH} \to 0.
\end{equation*}
\end{proof}

\subsection{Characterisation of Gaussian Measures Via Precision Operators}

\begin{lemma}
\label{le:PI}
Let $C_0 = (-\partial_t^2)^{-1}$ be the inverse of the Dirichlet Laplacian on $[-1,1]$ with domain $H^2([-1,1])\cap H^1_0([-1,1])$. Then $\mu_0=\cN(0,C_0)$ is
the distribution of a homogeneous Brownian bridge on $[-1,1]$. Consider measure $\nu \ll \mu_0$
defined by 
\begin{equation}
\frac{d \nu}{ d \mu_0} (x(\cdot) ) =\frac{1}{Z}\exp \Big(- \frac12 \int_{-1}^1 \theta(t) \, x(t)^2 \, dt  \Big) 
\end{equation} 
where $\theta$ is a smooth function with infimum strictly larger than -$\frac{\pi^2}{4}$ on $[-1,1]$.
Then $\nu$ is a centred Gaussian $\cN(0,C)$ with $C^{-1}=C_0^{-1}+\theta$.
\end{lemma}

The following proof closely follows techniques introduced
to prove Theorem 2.1 in \cite{pokern2012posterior}.

\begin{proof} 
As above, denote by $\cH = L^2([-1,1])$ and $\cH^1 =H^1_0([-1,1])$. Furthermore, let $(\ea, \lambda_\alpha, \alpha \geq 1)$ be the eigenfunction/eigenvalue pairs of $C_0$ ordered by decreasing eigenvalues.  For any $\gamma \geq 1$ let $\pg$ be the orthogonal projection on $\cH$ onto $\cH_\gamma = \spa(e_1, \ldots , e_\gamma)$. Denote by $\Hp=(\Id-\pg)\cH$.

For each $\gamma \geq 1$ define the measure $\nu_\gamma \ll \mu_0$ by

\begin{equation*}
\frac{d \nu_\gamma}{ d \mu_0} (x(\cdot) ) =\frac{1}{Z_\gamma}\exp \Big(- \frac12 \int_{-1}^1 \theta(t) \, \bigl(\pg x(t)\bigr)^2 \, dt  \Big).
\end{equation*}

We first show that the $\nu_\gamma$ are centred Gaussian and we 
characterize their covariance. To see this note that $\mu_0$
factors as the independent product of two Gaussians on $\Hm$ and
$\Hp$. Since the change of measure defining $\nu_\gamma$ depends
only on $\pg x \in \Hm$ it follows that $\nu_\gamma$ also factors as an independent
product. Furthermore, the factor on $\Hp$ coincides with the projection of  $\mu_0$ and is Gaussian.
On $\Hm$, which is finite dimensional,
it is clear that $\nu_\gamma$ is also Gaussian because the
change of measure is defined through a finite dimensional
quadratic form. This Gaussian is centred and has
inverse covariance (precision) given by $\pg(C_0^{-1}+\theta)\pg=\pg C^{-1}\pg.$  Hence $\nu_\gamma$ is also Gaussian; denote its covariance operator by $C_\gamma$.
  
A straightforward dominated convergence argument shows that $\nu_\gamma$
converges weakly to $\nu$ as a measure on $\cH$, and it follows 
that $\nu$ is a centred Gaussian by Lemma \ref{le:GaussianConv}; we
denote the covariance by $\Sigma$. It remains to show that 
$\Sigma=C.$
On the one hand, we have by Lemma \ref{le:GaussianConv}, item 3.),
that $C_\gamma$ converges to  $\Sigma$ in the operator norm. 
On the other hand we have for any $x \in \cH^1$ and for $\gamma \geq1$  that
\begin{align*}
\big| \langle x, C_\gamma^{-1} x\rangle - \langle x, C^{-1} x \rangle \big| 
& = \int_{-1}^1 \theta(t)   \big((\Id-\pg) x(t) \big)^2 \, dt \leq \| \theta \|_{L^\infty} \| (\Id-\pg) x(t) \|_{L^2}^2\\
& \leq   \| \theta \|_{L^\infty} \lambda_\gamma^2  \|  x(t) \|_{H^1_0}^2.
\end{align*} 
As the $\lambda_\gamma \to 0$ for $\gamma \to \infty$ and as  the operator $C^{\frac12} C_{0}^{-\frac12}$ is a bounded invertible operator on $\cH^1$ this implies the  convergence of $C_\gamma^{-1}$ to $C^{-1}$ in the strong resolvent sense by \cite[Theorem VIII.25]{RS1}. The conclusion then follows  as in the proof of Theorem \ref{thm:existence2}.



\end{proof}

\bibliographystyle{alpha}
\bibliography{KL}

\newcommand{\etalchar}[1]{$^{#1}$}
\def\cprime{$'$}
\begin{thebibliography}{ACOST07}

\bibitem[ACOST07]{archambeau2007gaussian}
C.~Archambeau, D.~Cornford, M.~Opper, and J.~Shawe-Taylor.
\newblock Gaussian process approximations of stochastic differential equations.
\newblock {\em Journal of Machine Learning Research}, 1:1--16, 2007.

\bibitem[AGS08]{AGS}
Luigi Ambrosio, Nicola Gigli, and Giuseppe Savar{\'e}.
\newblock {\em Gradient flows in metric spaces and in the space of probability
  measures}.
\newblock Lectures in Mathematics ETH Z\"urich. Birkh\"auser Verlag, Basel,
  second edition, 2008.

\bibitem[AOS{\etalchar{+}}07]{archambeau2007variational}
C.~Archambeau, M.~Opper, Y.~Shen, D.~Cornford, and J.~Shawe-Taylor.
\newblock Variational inference for diffusion processes.
\newblock {\em Pascal Network e-print}, 2007.

\bibitem[Bil09]{billingsley2009convergence}
Patrick Billingsley.
\newblock {\em Convergence of Probability Measures}.
\newblock John Wiley \& Sons, 2009.

\bibitem[BK87]{ball1987numerical}
JM~Ball and G~Knowles.
\newblock A numerical method for detecting singular minimizers.
\newblock {\em Numerische Mathematik}, 51(2):181--197, 1987.

\bibitem[BN06]{bishop2006pattern}
Christopher~M Bishop and Nasser~M Nasrabadi.
\newblock {\em {Pattern Recognition and Machine Learning}}, volume~1.
\newblock {Springer New York}, 2006.

\bibitem[Bog98]{Bog}
Vladimir~I. Bogachev.
\newblock {\em Gaussian measures}, volume~62 of {\em Mathematical Surveys and
  Monographs}.
\newblock American Mathematical Society, Providence, RI, 1998.

\bibitem[Csi75]{CS}
I.~Csisz{\'a}r.
\newblock {$I$}-divergence geometry of probability distributions and
  minimization problems.
\newblock {\em Ann. Probability}, 3:146--158, 1975.

\bibitem[CT12]{cover2012elements}
Thomas~M. Cover and Joy~A. Thomas.
\newblock {\em Elements of Information Theory}.
\newblock John Wiley \& Sons, 2012.

\bibitem[DE97]{DE}
Paul Dupuis and Richard~S. Ellis.
\newblock {\em A weak convergence approach to the theory of large deviations}.
\newblock Wiley Series in Probability and Statistics: Probability and
  Statistics. John Wiley \& Sons Inc., New York, 1997.
\newblock A Wiley-Interscience Publication.

\bibitem[DLSV13]{DLSV13}
M.~Dashti, K.J.H. Law, A.M. Stuart, and J.~Voss.
\newblock Map estimators and posterior consistency in bayesian nonparametric
  inverse problems.
\newblock {\em Inverse Problems}, 29:095017, 2013.

\bibitem[DPZ92]{ZDP}
G.~Da~Prato and J.~Zabczyk.
\newblock {\em Stochastic Equations in Infinite Dimensions}, volume~44 of {\em
  Encyclopedia of Mathematics and its Applications}.
\newblock Cambridge University Press, 1992.

\bibitem[EHN96]{engl1996regularization}
Heinz~Werner Engl, Martin Hanke, and Andreas Neubauer.
\newblock {\em Regularization of inverse problems}, volume 375.
\newblock Springer, 1996.

\bibitem[F{\"U}04]{FU}
D.~Feyel and A.~S. {\"U}st{\"u}nel.
\newblock Monge-{K}antorovitch measure transportation and {M}onge-{A}mp\`ere
  equation on {W}iener space.
\newblock {\em Probab. Theory Related Fields}, 128(3):347--385, 2004.

\bibitem[GM12]{giannakis2012quantifying}
D.~Giannakis and AJ~Majda.
\newblock Quantifying the predictive skill in long-range forecasting. part i:
  Coarse-grained predictions in a simple ocean model.
\newblock {\em Journal of Climate}, 25(6), 2012.

\bibitem[{Hai}09]{martinSPDE}
M.~{Hairer}.
\newblock {An Introduction to Stochastic PDEs}.
\newblock {\em ArXiv e-prints}, July 2009.

\bibitem[HSV11]{HSV11}
M.~Hairer, A.M. Stuart, and J.~Voss.
\newblock Signal processing problems on function space: {B}ayesian formulation,
  stochastic {PDEs} and effective {MCMC} methods.
\newblock In D.~Crisan and B.~Rozovsky, editors, {\em The Oxford Handbook of
  Nonlinear Filtering}, pages 833--873. Oxford University Press, 2011.

\bibitem[KP13]{katsoulakis2013information}
MA~Katsoulakis and P~Plech{\'a}{\v{c}}.
\newblock Information-theoretic tools for parametrized coarse-graining of
  non-equilibrium extended systems.
\newblock {\em The Journal of Chemical Physics}, 139(7):074115, 2013.

\bibitem[KPT07]{katsoulakis2007coarse}
MA~Katsoulakis, L~Plech{\'a}c, Pand Rey-Bellet, and DK~Tsagkarogiannis.
\newblock Coarse-graining schemes and a posteriori error estimates for
  stochastic lattice systems.
\newblock {\em ESAIM: Mathematical Modelling and Numerical Analysis},
  41(3):627--660, 2007.

\bibitem[KS05]{kaipio2005statistical}
Jari~P Kaipio and Erkki Somersalo.
\newblock {\em Statistical and Computational Inverse Problems}, volume 160.
\newblock Springer, 2005.

\bibitem[LG13]{LeGall}
Jean-Fran{\c{c}}ois Le~Gall.
\newblock {\em Mouvement {B}rownien, Martingales et Calcul Stochastique},
  volume~71 of {\em Math\'ematiques et Applications}.
\newblock Springer, 2013.

\bibitem[Liu08]{liu2008monte}
Jun~S Liu.
\newblock {\em Monte Carlo Strategies in Scientific Computing}.
\newblock Springer, 2008.

\bibitem[McC97]{McCann}
Robert~J. McCann.
\newblock A convexity principle for interacting gases.
\newblock {\em Adv. Math.}, 128(1):153--179, 1997.

\bibitem[MG11]{majda2011improving}
AJ~Majda and B~Gershgorin.
\newblock Improving model fidelity and sensitivity for complex systems through
  empirical information theory.
\newblock {\em Proceedings of the National Academy of Sciences},
  108(25):10044--10049, 2011.

\bibitem[PSSW14]{PSSW14}
F.~Pinski, G.~Simpson, A.M. Stuart, and H.~Weber.
\newblock Numerical methods based on {K}ullback-{L}eibler approximation for
  probability measures.
\newblock {\em preprint}, 2014.

\bibitem[PSVZ12]{pokern2012posterior}
Y~Pokern, AM~Stuart, and JH~Van~Zanten.
\newblock Posterior consistency via precision operators for bayesian
  nonparametric drift estimation in sdes.
\newblock {\em Stochastic Processes and their Applications}, 123:603--628,
  2012.

\bibitem[RS75]{RS2}
Michael Reed and Barry Simon.
\newblock {\em Methods of Modern Mathematical Physics. {II}. {F}ourier
  Analysis, Self-adjointness}.
\newblock Academic Press, New York, 1975.

\bibitem[RS80]{RS1}
Michael Reed and Barry Simon.
\newblock {\em Methods of Modern Mathematical Physics. {I}}.
\newblock Academic Press Inc. [Harcourt Brace Jovanovich Publishers], New York,
  second edition, 1980.
\newblock Functional analysis.

\bibitem[Stu10]{S10a}
A.M. Stuart.
\newblock Inverse problems: a {B}ayesian perspective.
\newblock In {\em Acta Numerica 2010}, volume~19, page 451. 2010.

\bibitem[Vil09]{Villani}
C{\'e}dric Villani.
\newblock {\em Optimal Transport}, volume 338 of {\em Grundlehren der
  Mathematischen Wissenschaften [Fundamental Principles of Mathematical
  Sciences]}.
\newblock Springer-Verlag, Berlin, 2009.
\newblock Old and new.

\end{thebibliography}
\end{document}